\documentclass[notitlepage,twoside,a4paper]{amsart}
\usepackage{mathrsfs}
\usepackage{amsmath,amssymb,enumerate}
\usepackage{epsfig,fancyhdr,color}
\usepackage{epstopdf}
\usepackage{amssymb}
\usepackage{amsmath,amsthm}
\usepackage{latexsym}
\usepackage{amscd}
\usepackage{psfrag}
\usepackage{graphicx}
\usepackage{epsf}
\usepackage[latin1]{inputenc}
\usepackage[all]{xy}
\usepackage{tikz}
\usepackage{prettyref}
\usepackage{subfigure}
\usepackage{float}
\usepackage{color}
\usepackage{array}
\usepackage{cite}
\usepackage[totalwidth=16cm,totalheight=24cm]{geometry}

\newcommand{\II}{I\hspace{-0.1cm}I}
\newcommand{\III}{I\hspace{-0.1cm}I\hspace{-0.1cm}I}

\newtheorem{theorem}{\rm\bf Theorem}[section]
\newtheorem{proposition}[theorem]{\rm\bf Proposition}
\newtheorem{conjecture}[theorem]{\rm\bf Conjecture}
\newtheorem{lemma}[theorem]{\rm\bf Lemma}

\newtheorem{definition}[theorem]{\rm\bf Definition}

\newtheorem{question}[theorem]{\rm\bf Question}

\theoremstyle{remark}
\newtheorem{example}[theorem]{\rm\bf Example}

\newtheoremstyle{named}{}{}{\itshape}{}{\bfseries}{.}{.5em}{#1 \thmnote{#3}}
\theoremstyle{named}
\newtheorem*{namedquestion}{Question}

\newcommand{\C}{{\mathbb C}}
\newcommand{\CP}{{\mathbb CP}}
\newcommand{\N}{{\mathbb N}}
\newcommand{\DD}{{\mathbb D}}
\newcommand{\HH}{{\mathbb H}}
\newcommand{\R}{{\mathbb R}}
\newcommand{\RP}{{\mathbb {RP}}}
\newcommand{\HP}{{\mathbb {HP}}}
\newcommand{\Z}{{\mathbb Z}}

\newcommand{\cC}{{\mathcal C}}
\newcommand{\cE}{{\mathcal E}}
\newcommand{\cL}{{\mathcal L}}

\newcommand{\cT}{{\mathcal T}}
\newcommand{\cML}{{\mathcal{ML}}}
\newcommand{\cCP}{{\mathcal{CP}}}
\newcommand{\cQS}{{\mathcal{QS}}}

\newcommand{\dS}{\mathbb{DS}}
\newcommand{\AdS}{\mathbb{ADS}}

\newcommand{\tr}{\mbox{tr}}

\newcounter{notes}%

\def\interieur#1{\mathord{\mathop{\kern 0pt #1}\limits^\circ}}

\title[The Weyl problem for complete surfaces]{On the Weyl problem for complete surfaces in the hyperbolic and anti-de Sitter spaces}

\author{Jean-Marc Schlenker}
\address{Jean-Marc Schlenker:
University of Luxembourg, FSTM, Department of Mathematics, 
Maison du nombre, 6 avenue de la Fonte,
L-4364 Esch-sur-Alzette, Luxembourg}
\email{jean-marc.schlenker@uni.lu}
\thanks{Partially supported by FNR project O20/14766753.}

\date{v2, \today}

\begin{document}

\begin{abstract}
  The classical Weyl problem (solved by Lewy, Alexandrov, Pogorelov, and others) asks whether any metric of curvature $K\geq 0$ on the sphere is induced on the boundary of a unique convex body in $\R^3$. The (positive) answer was extended to surfaces in hyperbolic space by Alexandrov in the 1950s, and a ``dual'' statement, describing convex bodies in terms of the third fundamental form of their boundary (e.g. their dihedral angles, for an ideal polyhedron) was later proved.

  We describe two conjectural generalizations of the Weyl problem in $\HH^3$ and its dual to unbounded convex subsets and convex surfaces, in ways that are relevant to contemporary geometry since a number of recent results and well-known open problems can be considered as special cases. One focus is on unbounded convex domain in $\HH^3$. The boundary data is then the conformal structure on the ``full'' boundary -- including the conformal structure at infinity, but also the conformal structure of the induced metric (resp. third fundamental form) on the boundary in $\HH^3$, together with the induced metric (resp. third fundamental form) in $\HH^3$. 
  
  A second direction is towards complete, locally convex disks of infinite area immersed in $\HH^3$ and surfaces in hyperbolic ends -- with connections to questions on circle packings or grafting.

  Similar statements are proposed in anti-de Sitter geometry, a Lorentzian cousin of hyperbolic geometry where interesting new phenomena can occur, and in Minkowski and Half-pipe geometry. We also collect some partial new results mostly based on recent works by different authors. 

\end{abstract}

\maketitle

\tableofcontents
        
\section{The classical Weyl problem in Euclidean and hyperbolic space}

\subsection{Convex surfaces in Euclidean space}
\label{ssc:weyl-e}
Smooth or polyhedral convex surfaces in the 3-dimensional Euclidean space $\R^3$ have been a staple of geometry throughout the ages. One of the first ``modern'' results on this theme was the proof by Legendre \cite{legendre} and Cauchy \cite{cauchy} of the rigidity of convex polyhedra: if two polyhedra have the same combinatorics and corresponding faces are isometric, then they are congruent.

Much later, Alexandrov \cite{alex} realized that Cauchy's rigidity result is actually a consequence of a much deeper statement. The induced metric on the boundary of a convex polyhedron in $\R^3$ is an Euclidean metric with cone singularities with cone angle less than $2\pi$ at the vertices. Alexandrov proved that each metric of this type on the sphere can be realized as the induced metric on the boundary of a unique polyhedron. This result was later extended by Alexandrov \cite{alex} to bounded convex subset of $\R^3$, thus solving the ``Weyl problem'', proposed by H. Weyl in 1915 and for which substantial progress had already been made by H. Lewy \cite{lewy1935priori} and others.

\begin{theorem}[Lewy, Alexandrov] \label{tm:alexE}
  Any smooth metric of positive curvature on the sphere is induced on the boundary of a unique smooth strictly convex subset of $\R^3$. 
\end{theorem}

Here (and elsewhere) by ``strictly convex'' we mean that the shape operator of the boundary is positive definite. By ``smooth'' we mean $C^\infty$, although different degrees of regularity could be considered.

This result was then further extended by Pogorelov \cite{Po} to non-smooth metrics: the induced metric on the boundary of any bounded convex subset of $\R^3$, even if non-smooth, is of curvature $K\geq 0$ in the sense of Alexandrov, and each such metric on the sphere is realized on the boundary of a unique bounded convex body.

\subsection{The Weyl problem and its dual for closed convex surfaces in hyperbolic space}
\label{ssc:weyl-h}

Alexandrov and Pogorelov extended Theorem \ref{tm:alexE} to hyperbolic space.

\begin{theorem}[Alexandrov, Pogorelov] \label{tm:alexH}
  Any smooth metric of curvature $K> -1$ on the sphere is induced on the boundary of a unique convex subset of $\HH^3$ with smooth boundary. 
\end{theorem}

Note that Theorem \ref{tm:alexE} can be seen as a limit case of Theorem \ref{tm:alexH}. If $h$ is a smooth metric on $S^2$ of curvature $K>0$, one can apply Theorem \ref{tm:alexH} to the sequence of metrics $h_n=(1/n)h$, which are also metrics of curvature $K>0>-1$, yielding a sequence of surfaces $S_n\subset \HH^3$, with the induced metric on $S_n$ isometric to $(1/n)h_n$. The diameter of $S_n$ converges to $0$ as $n\to 0$, so, after applying a sequence of isometries, the sequence $(S_n)_{n\in \N}$ converges to a point $x_0\in \HH^3$. After applying to $\HH^3$ a squence of homotheties (multiplication of the metric by $n$) and extracting a subsequence, $(S_n)_{n\in \N}$ converges to a convex surfaces in $\R^3$, with induced metric $h$.

Analogs of Theorem \ref{tm:alexH} also hold for ideal polyhedra \cite{rivin-comp} and for hyperideal polyhedra \cite{shu}. The existence of those extensions hint at the possibility of a wider extension to more general unbounded convex domains in $\HH^3$ as considered here.

\medskip

In addition, a new phenomenon appears, which was present in the Euclidean situation only in a rather degenerate form.  Recall that if $S$ is a surface in $\HH^3$ with induced metric $I$ and shape operator $B$, its third fundamental form is defined by
$$ \III(x,y)=I(Bx,By)~, $$
for any two vector fields $x,y$ tangent to $S$. The third fundamental form therefore measures how ``curved'' the surface $S$ is. The following statement is from \cite{these}.

\begin{theorem} \label{tm:these}
  Let $\Omega$ be a subset of $\HH^3$ with smooth strictly convex boundary. The third fundamental form of the boundary has $K<1$, and its closed geodesics have length $L>2\pi$. Conversely, any metric of this type on the sphere can be realized as the third fundamental form of the boundary of a unique subset of $\HH^3$ with smooth, strictly convex boundary.
\end{theorem}

A version of Theorem \ref{tm:these} was obtained earlier by Hodgson and Rivin \cite{HR} for convex polyhedra, following work of Andreev \cite{Andreev-ideal}. Note that taken together, the condition that $K<1$ and that closed geodesics have length $L>2\pi$ is (almost) equivalent to asking that the metric is (globally) ``strongly'' $CAT(1)$, and it is this strong $CAT(1)$ condition that we will use below for simplicity.

Similar statements exist for the dihedral angles (or more precisely ``dual metrics'', which are essentially polyhedral versions of the third fundamental form) of compact hyperbolic polyhedra \cite{Andreev,HR}, for the dihedral angles of ideal polyhedra \cite{Andreev-ideal,rivin-annals} and of hyperideal polyhedra \cite{bao-bonahon,rousset1,shu}.

\bigskip

In the next two subsection, we introduce two different questions, each stated for the induced metric and, in a dual version, for the third fundamental form. In the next sections we will describe special cases of those problems corresponding either to noted open problems, or to relatively recent results.

There are in fact two separate points of views on the classical Weyl problem, leading to two different kinds of generalizations.
\begin{itemize}
\item It can be considered as a statement on the induced metric on the boundary of convex subsets of $\R^3$ (or $\HH^3$). This point of view leads to Questions $W^{\Omega}_{\HH^3}$ and its duals.
  \item Or, it can be seen as a statement on isometric embeddings in $\R^3$ (or $\HH^3$) of surfaces with Riemannian metrics, under curvature conditions implying that the image surfaces are locally convex. This point of view leads to Question $W^{imm}_{\HH^3}$, which concerns immersed surfaces in $\HH^3$ with prescribed induced metrics (or its dual concerning immersed surfaces with prescribed third fundamental form). 
\end{itemize}

\subsection{Unbounded convex domains in $\HH^3$}
\label{ssc:intro-domain}

In this subsection we consider a convex subset $\Omega\subset \HH^3$. We first introduce a suitable notion of boundary.

\begin{definition} \label{df:boundary}
  Given an (unbounded) convex subset $\Omega\subset \HH^ 3$, we denote by
  \begin{itemize}
  \item $\partial \Omega$, the boundary of $\Omega$ in $\HH^3$, 
  \item $\partial_\infty \Omega$ its ideal boundary, which can be defined for instance using the projective model of $\HH^3$, as the intersection of $\partial_\infty \HH^3$ with the closure of $\Omega$ in the projective model,
  \item $\partial_0\Omega=\partial\Omega\cup\partial_\infty\Omega$, the ``full'' boundary of $\Omega$.
  \end{itemize}
\end{definition}

There are several examples that can come to mind, and that occur quite naturally in different situations.

\begin{example} \label{ex:quasicircle}
  $\Omega$ could be a convex subset in $\HH^3$ with boundary at infinity a quasi-circle. This case is considered, when $\partial\Omega$ has constant curvature, in \cite{convexhull}. If $M$ is a quasifuchsian manifold, and $C\subset M$ is a compact, geodesically convex subset -- that is, any geodesic segment in $M$ with endpoints in $C$ is contained in $C$ -- then the universal cover $\tilde{C}$ of $C$ is a subset of $\HH^3$ of this type, with $\partial_\infty\tilde{C}$ equal to the limit set of the holonomy representation of $M$, which is a quasicircle in $\partial_\infty \HH^3=\C P^1$.  
\end{example}

\begin{example} \label{ex:quasidisk}
  $\Omega$ could be a convex subset in $\HH^3$ with boundary at infinity $\partial_\infty\Omega$ a quasidisk, and boundary in $\HH^3$ a complete disk in $\HH^3$ with asymptotic boundary $\partial_\infty(\partial \Omega)=\partial(\partial_\infty\Omega)$. This case occurs for instance as the universal cover of a domain $\tilde{C}$ of a geodesically convex domain $C$ in a quasifuchsian manifold $M$, when $\partial C$ is a connected locally convex surface in $M$ and $C$ contains a neighborhood of a connected component of the asymptotic boundary of $M$.
\end{example}

\begin{example} \label{ex:cc}
  Let $M$ be a convex co-compact hyperbolic manifold with incompressible boundary, and let $C\subset M$ be a compact, geodesically convex subset with smooth boundary. Then $\tilde{C}$ is a convex domain in $\HH^3$. Each connected component of $\partial C$ lifts to at least one complete disk in $\HH^3$, but in most cases each connected component of $\partial C$ lifts to infinitely many such disks. In this case, $\partial_\infty\tilde{C}$ is the limit set of the holonomy representation of $M$. If now $C\subset M$ is non-compact but contains a neighborhood of one connected component of the asymptotic boundary of $M$, then $\partial_\infty\tilde{C}$ contains at least one quasidisk (but in most cases infinitely many) which are lifts of this connected component of $\partial_\infty M$. 
\end{example}

\subsection{Conformal structures on $\partial_0\Omega$}

Let $\Omega\subset \HH^3$ be a convex subset. The statements considered below are typically in terms of a conformal (or complex) structure on $\partial_0\Omega$. We understand this notion in the following manner. Note that $\partial_0\Omega$ is naturally equipped with a topology, coming from that of the projective (or Poincar\'e disk) model of $\HH^3$. 

\begin{definition} \label{df:I}
  Let $\Omega\subset \HH^3$ be a convex domain. An {\em induced metric conformal structure} on $\partial_0\Omega$ is a homeomorphism $\phi:\partial_0\Omega \to \C P^1$, which is conformal on $\partial_\infty \Omega$ (equipped with its standard conformal structure as a subset of $\partial_\infty\HH^3$), and conformal on $\partial\Omega$, equipped with its metric induced from that of $\HH^3$.  
\end{definition}


The existence and uniqueness of an induced metric conformal structure on the ``full'' boundary of a general convex subset in $\HH^3$ is not clear.

\begin{question}
Let $\Omega\subset \HH^3$ be a convex subset. Is there necessarily an induced metric conformal structure on $\partial_0\Omega$? Is it always unique (up to post-composition by a M\"obius homeomorphism of $\CP^1$)? 
\end{question}

Lemma \ref{lm:existence-confI} shows that existence holds in the cases which we will consider as a first step. It is proved in Section  \ref{ssc:conformalI}. To state it, we need a ``technical'' condition.

\begin{definition}
  Let $h$ be a complete, smooth, conformal metric on the disk $\DD$. We will say that $h$ has {\em bounded derivatives} if, when $h$ is written as $h=e^{2u}h_{-1}$, with $h_{-1}$ the complete, conformal hyperbolic metric on $\DD$, there exists for each $k\geq 0$ a constant $C_k>0$ that bounds all $k$-th derivatives of $u$ with respect to $h_{-1}$.
\end{definition}

\begin{lemma} \label{lm:existence-confI}
  \begin{enumerate}
  \item Assume that there exists $k_0>0$ such that $\partial\Omega$ has principal curvatures at most $k_0$. Then $\partial_0\Omega$ admits a unique induced metric conformal structure.
  \item This is in particular the case if there exist $\epsilon'>0$ and $k>0$ such that
    \begin{itemize}
    \item $\partial\Omega$ is a disjoint union of topological disks, 
    \item the induced metric on $\partial\Omega$ has curvature $K\in [-1+\epsilon', 1/\epsilon']$ and bounded derivatives.
    \end{itemize}
  \end{enumerate}
\end{lemma}

Suppose $\phi:\partial_0\Omega\to \CP^1$ is an induced metric conformal structure for $\Omega$. Assume that $\phi(\Lambda_\Omega)$ is {\em conformally rigid}, in the sense that any homeomorphism $\psi:\CP^1\to \CP^1$ which is conformal outside of $\phi(\Lambda_\Omega)$ is M\"obius. Then $\phi$ is unique, up to left composition by a M\"obius transformation. For instance, quasicircles in $\CP^1$ are conformally rigid, see e.g. \cite{bishop:some}.

\medskip

There is a dual notion of conformal structure for the third fundamental form. In case $\Omega$ does not have smooth boundary, its third fundamental form might not be well-defined as a metric on $\partial\Omega$, so one needs to consider instead the dual surface in the de Sitter space $\partial\Omega^*$, see Section \ref{ssc:duality-hyperbolic}.

\begin{definition} \label{df:III}
  Let $\Omega\subset \HH^3$ be a convex domain. A {\em third fundamental form conformal structure} on $\partial_0\Omega$ is a homeomorphism $\phi:\partial_0\Omega^* \to \C P^1$, which is conformal on $\partial_\infty \Omega=\partial_\infty\Omega^*$ (equipped with its standard conformal structure as a subset of $\partial_\infty\HH^3$), and conformal on $\partial\Omega^*$, equipped with its metric induced from that of $dS^3$.  
\end{definition}

We state and prove in Section \ref{ssc:conformalIII} the analog of Lemma \ref{lm:existence-confI} for the third fundamental form, as well as a uniqueness statement.

\subsection{Boundary data on convex subset of $\HH^3$}

In this section we state a question on the induced metric on the boundary of a (possibly unbounded) convex subset in $\HH^3$, as a well as a dual question concerning the third fundamental form. We provide partial answers in special cases.
The main questions are stated here in a rather general manner, so that it's not clear whether a positive answer can be given in that generality, or whether stronger hypothesis are needed. 

\begin{namedquestion}[$W^{\Omega}_{\HH^3}$] \label{q:Omega}
  Let $U\subset \C P^1$ be an open subset, and let $h$ be a complete metric on $U$, of curvature $K>-1$. 
  Is there a unique convex domain $\Omega\subset \HH^3$ with an induced metric conformal structure $\phi:\partial_0\Omega\to \C P^1$ such that $\phi(\partial \Omega)=U$ and that $\phi^*h=I$ on $\partial\Omega$?
\end{namedquestion}

For instance, when $U=\CP^1$, the answer is positive and corresponds to the solution of the classical hyperbolic Weyl problem. 

\medskip

We can state a very similar dual question.

\begin{namedquestion}[$W^{* \Omega}_{\HH^3}$]
  Let $U\subset \C P^1$ be an open subset, and let $h^*$ be a complete metric on $U$, of curvature $K<1$, with closed geodesics of length $L>2\pi$. 
  Is there a unique convex domain $\Omega\subset \HH^3$ with an third fundamental form conformal structure $\phi:\partial_0\Omega^*\to \C P^1$ such that $\phi(\partial \Omega^*)=U$ and that $\phi^*h^*=I^*$ on $\partial\Omega$?
\end{namedquestion}

Here and in the other statements below, the uniqueness is up to the action of the group of isometries of the ambient space, here $\HH^3$. 

We will see in Section \ref{sc:Omega} a number of relations between questions $W^{\Omega}_{\HH^3}$ and $W^{\Omega *}_{\HH^3}$ on the one hand, and open questions or known results concerning quasifuchsian or convex co-compact hyperbolic manifolds. We will also mention recent results which can be considered as answers to the existence part of those questions, concerning the geometric data on the convex core of quasifuchsian (or more generally convex co-compact) hyperbolic manifolds.


\subsection{Complete locally convex surfaces in $\HH^3$}
\label{ssc:complete}

So far, we have considered Theorem \ref{tm:alexH} and its dual Theorem \ref{tm:these} as statements on the induced metric or third fundamental forms on the boundary of convex subsets in $\HH^3$. There is, however, another way to consider them, namely, as statements on isometric immersions of locally convex surfaces (resp. on embeddings inducing a given third fundamental form). For closed surfaces in $\HH^3$, the two points of view are equivalent, since any closed locally convex immersed surface bounds a convex domain. For unbounded surfaces, however, the two points of view are quite different, since a locally convex immersion of a surface, even if proper, does not need to be an embedding, and its image might not bound a convex subset. (A simple example can be obtained by deforming the universal cover of the set of points equidistant from a geodesic in $\HH^3$.)

Given a complete, locally convex immersed disk $S\subset \HH^3$, there is in general no meaningful way to see it as a part of the boundary of a convex subset. However, local convexity means that the hyperbolic Gauss map --- the map sending a point $x\in S$ to the endpoint at infinity of the geodesic starting from $x$ in the direction of the normal vector pointing towards the concave side of $S$ --- is a local homeomorphism. 

One can therefore consider on $S$ the pull-back of the conformal structure at infinity $c_\infty$ by this hyperbolic Gauss map $G:S\to \CP^1$, and the gluing map at infinity between the induced metric (resp. third fundamental form) on $S$ and this pull-back conformal structure. Specifically, given a proper locally convex embedding $\phi:\DD\to \HH^3$, let $U_0:\DD\to \DD$ the Riemann uniformization map for the pull-back $(G\circ\phi)^*c_\infty$ by the Gauss map of the conformal metric at infinity, and let $U_1:\DD\to \DD$ be the Riemann uniformization map for the conformal class of the induced metric $I$ on $S$. We call $\partial (U_1^{-1}\circ U_0):\partial \DD\to \partial \DD$ the {\em induced metric gluing map} for $\phi$. We can of course define in a similar manner the {\em third fundamental gluing map}. Both are well-defined up to pre- and post-composition by M\"obius transformations.

This definition of a gluing map at infinity can be extended to a wider setting, where it will be used below. Let $g_0, g_1$ be two conformal Riemannian metrics on a surface $S$ homeomorphic to the disk, both conformal to the disk $\DD$ (i.e. of hyperbolic type). Let $U_i:\DD \to (S, g_i)$, $i=0,1$, be the Riemann uniformization map. The {\em gluing map at infinity} between $g_0$ and $g_1$ is the homeomorphism $\partial (U_1^{-1}\circ U_0):\partial \DD\to \partial \DD$. It is well-defined up to pre- and post-composition by M\"obius transformations. The same definition can be used in $g_0$ or $g_1$ is replaced by a conformal structure (also conformal to the disk) on $S$.

\begin{namedquestion}[$W^{imm}_{\HH^3}$]
  Let $h$ be a complete metric of curvature $K\in [-1,-\epsilon]$ on the disk $\DD$, for some $\epsilon>0$, and let $\sigma:\partial \DD\to \partial \DD$ be a quasi-symmetric homeomorphism. Is there a unique locally convex isometric immersion $\phi:(S,h)\to \HH^3$ such that the induced metric gluing map of $\phi$ is $\sigma$ (up to M\"obius transformations)?
\end{namedquestion}

\begin{namedquestion}[$W^{* imm}_{\HH^3}$]
  Let $h$ be a complete metric of curvature $K\in [-1/\epsilon,-\epsilon]$ on the disk $\DD$, for some $\epsilon>0$, with closed geodesics of length $L>2\pi$, and let $\sigma:\partial \DD\to \partial\DD$ be a quasi-symmetric homeomorphism. Is there a unique immersion $\phi:(S,h)\to \HH^3$ such that the pull-back by $\phi$ of the third fundamental form is $h$, and that the third fundamental form gluing map of $\phi$ is $\sigma$ (up to M\"obius transformations)? 
\end{namedquestion}

Note that the flavor of those questions is quite different from those of Question $W^{\Omega}_{\HH^3}$ and its dual, since the conformal metric at infinity that is considered here is on the {\em concave} side of the surface, rather than on the convex side as for Question $W^{\Omega}_{\HH^3}$ and its dual.

There is a similarity between those two questions and the ``hyperbolic Plateau problems'' considered for $K$-surfaces in $\HH^3$, see e.g. \cite{labourie:morse,labourie-random,smith:hyperbolic,smith:pointed}.

We will see in Section \ref{sc:immersions} that answers to special cases of Questions $W^{imm}_{\HH^3}$ and $W^{* imm}_{\HH^3}$ are directly related to recent results on the grafting map, while a beautiful conjecture on circle packings can be seen as another special case of Question $W^{*imm}_{\HH^3}$.

There are also other statements, similar to Question $W^{imm}_{\HH^3}$, that can be proved for $K$-surfaces but might extend to surfaces of variable curvature, see Section \ref{ssc:similar}. One concerns the gluing map at infinity between the induced metric on a surface and its third fundamental form, while the other is about the gluing map at infinity between the induced metric (or third fundamental form) of a surface and the conformal class of $(1+K)I+\III$.

\subsection{Anti-de Sitter geometry}

The 3-dimensional anti-de Sitter space $\AdS^3$ is the model 3-dimensional Lorentzian space of constant curvature $-1$. It can be considered as the Lorentzian cousin of $\HH^3$. We refer to Section \ref{ssc:ads-basics}, or to \cite{mess,mess-notes,bonsante-seppi:anti} for the basic properties of AdS geometry, and of space-like surfaces in $\AdS^3$. The boundary at infinity $\partial_\infty\AdS^3$ is identified with $\RP^1\times \RP^1$, and equipped with a conformal Lorentz structure.

It is quite natural to consider convex subsets of $\AdS^3$ with boundary at infinity an acausal curve in $\partial \AdS^3$. We are specifically interested in convex subsets with space-like boundary having a boundary ``at infinity'' which is a ``quasicircle'', which arise for instance as the universal covers of globally hyperbolic compact maximal AdS spacetimes. (See Section \ref{ssc:ads-basics} for the definition of a quasicircle in $\partial \AdS^3$.) Given such a convex subset  $\Omega$, its boundary in $\AdS^3$ is the disjoint union of two disks $D_-$ and $D_+$, each equipped with an induced metric which we require to be complete. Moreover there is a natural identification between the conformal boundaries of $D_-$ and $D_+$ -- each equipped with its induced metric -- given by the gluing at infinity of the two, we call this map the {\em gluing map at infinity} between $D_-$ and $D_+$, each equipped with its induced metric.

The Gauss equation in $\AdS^3$ indicates that the induced metrics on $D_-$ and $D_+$ have curvature $K\leq -1$. As in the hyperbolic case, the induced metric on its boundary is not sufficient to recover a convex subset with boundary at infinity a quasicircle, but we expect that this induced metric together with a ``gluing'' data at infinity uniquely characterizes it.


\begin{namedquestion}[$W^{\Omega}_{\AdS^3}$]
  Let $g_-, g_+$ be two complete Riemannian metrics on the disk $\DD$, of curvature $K\in [-1/\epsilon, -1-\epsilon]$ for some $\epsilon>0$. Let $\sigma:\partial (\DD, g_-) \to \partial (\DD, g_+)$ be a quasisymmetric homeomorphism. Is there a unique convex subset $\Omega\subset \HH^3$, with boundary at infinity  $\partial_\infty \Omega$ a quasicircle, such that the induced metric on the past (resp. future) boundary component of $\partial\Omega$ is isometric to $g_-$ (resp. $g_+$), with induced metric gluing map equal to $\sigma$? 
\end{namedquestion}

Note that the boundaries $\partial (\DD, g_\pm)$ appearing here are the conformal boundaries of $(\DD, g_-)$ and $(\DD, g_+)$. Each is equipped with a real projective structure through Riemann uniformization, so that the notion of quasi-symmetric homeomorphism makes sense.

One could of course state a dual question, concerning the third fundamental form on the past and future boundary components of a convex subset, and the corresponding gluing map. This dual question, however, is equivalent to Question $W^{thin}_{\AdS^3}$ through a duality between convex subsets of $\AdS^3$, see Section \ref{sssc:duality}.

Recently, A. Mesbah \cite{mesbah2024} has given a positive answer to the existence part of this question under slightly stronger regularity asumptions, see Section \ref{sssc:existence}. (He assumes that the two metrics have all derivatives uniformly bounded with respect to the hyperbolic metric in the same conformal class.) Moreover, an existence statement holds for the measured bending lamination on the boundary of the convex hull of a quasicircle in $\partial \AdS^3$. 

\medskip

We will recall in Section \ref{ssc:ads-basics} how the left (or right) metric can be defined for any space-like surface of negative curvature in $\AdS^3$. The case of globally hyperbolic AdS spacetimes (see below), and in particular the Mess analog of the Bers Simultaneous Uniformization for those spacetimes \cite{mess}, suggest that the left and right metrics play the role of the conformal metrics at infinity in hyperbolic geometry. We can therefore state an AdS analog of Question $W^{imm}_{\HH^3}$.

Given a space-like, complete surface $S\subset \AdS^3$, one can consider on $S$ the induced metric $I$, as well as the left metric $m_L$ (using the definition of the left metric in \cite{minsurf}). Again, we can consider the Riemann uniformization map $U_0:\DD\to S$ for the left metric $m_L$, and the Riemann uniformization map $U_1:\DD\to S$ for the induced metric. The boundary map $\partial(U_1\circ U_0^{-1})$ is then the {\em induced metric gluing map} of $S$.

\begin{namedquestion}[$W^{imm}_{\AdS^3}$]
  Let $h$ be a Riemannian metric on $\DD$ of bounded, uniformly negative curvature, and let $\sigma:\partial\DD\to \partial \DD$ be a quasisymmetric homeomorphism. Is there a unique locally convex properly embedded surface $S\subset \AdS^3$ such that the induced metric on $S$ is isometric to $h$, while its induced metric gluing map is $\sigma$?  
\end{namedquestion}

As above, the curvature conditions here could conceivably be relaxed. Another remark is that there is a dual statement, but it is equivalent to Question $W^{imm}_{\AdS^3}$ through the duality described in Section \ref{sssc:duality}. The same question can also of course be asked for the right metric. 

We will see in Section \ref{sc:ads} some cases where a positive answer can be given to Question $W^{imm}_{\AdS^3}$.
\begin{itemize}
\item For pleated surfaces, that is, when the metric $h$ has constant curvature $-1$, Question $W^{imm}_{\AdS^3}$ has a positive answer, equivalent to Thurston's Earthquake Theorem \cite{thurston:earthquakes,FLP}.
\item For $K$-surfaces (when the metric $h$ has constant curvature $K<-1$) the answer is also positive, and follows from a recent result of Bonsante and Seppi \cite{bonsante-seppi:area}. For $K$-surfaces invariant under a cocompact surface group, the positive answer follows from earlier work on ``landslides'' \cite{cyclic,cyclic2}.
\end{itemize}

There is in the physics literature a notion of globally hyperbolic maximal compact (GHMC) AdS spacetime. Mess discovered some remarkable analogies between quasifuchsian hyperbolic manifolds and those GHMC AdS spacetimes \cite{mess,mess-notes} which for this reason are now sometimes called ``quasifuchsian AdS spacetimes''. In this analogy, the role of the conformal metrics at infinity of a quasifuchsian hyperbolic manifold is played, for quasifuchsian AdS spacetimes, by the ``left'' and ``right'' hyperbolic metrics. Considering closed surfaces in quasifuchsian AdS spacetimes is equivalent to considering surfaces invariant under a co-compact surface group action in Question $W^{\Omega}_{\AdS^3}$. In that case, existence results are known. The results can then be stated in terms of the induced metrics or third fundamental forms on the boundary of convex domains in quasifuchsian AdS spacetimes. One can mention in particular:
\begin{itemize}
\item The result of Diallo (see \cite{diallo2013} or \cite[Appendix]{convexhull}): any pair of hyperbolic metrics on a surface of genus at least $2$ can be realized as the induced metric on the boundary of the convex core of a GHMC AdS spacetime.
\item Its extension by Tamburelli \cite{tamburelli2016} to prescribing metrics of curvature less than $-1$ on the boundary of convex domains in GHMC AdS spacetimes.
\item A dual result \cite{earthquakes} on prescribing the measured bending laminations on the boundary of the convex core of GHMC AdS spacetimes.
\end{itemize}

The analog of Question $W^{imm}_{\AdS^3}$ in this globally hyperbolic compact setting also has a positive answer, see \cite{immersed_weyl_ads}.

\begin{theorem}[Chen-S. 2024]
  Let $S$ be a closed surface of genus at least $2$, let $h$ be a complete Riemannian metric of curvature $K<-1$ on $S$, and let $h_0\in\cT_S$ be a hyperbolic metric on $S$. There is a unique quasifuchsian AdS spacetime $M$ with left metric isotopic to $h_0$ and containing a past-convex spacelike surface with induced metric isotopic to $h$.
\end{theorem}

More details can be found in Section \ref{sc:ads}.

\subsection{Minkowski geometry}

There is another setting where an analog of the Weyl problem can be considered for unbounded surfaces: the Minkowski space $\R^{2,1}$. There is a well-understood way to associate a surface in $\R^{2,1}$ to a first-order deformation of a hyperbolic plane in $\HH^3$ -- in fact one can associate to such a first-order deformation a surface in half-pipe geometry, and then by duality a surface in $\R^{2,1}$, see Section \ref{sc:minkowski}.

Through this correspondence, the following question turns out to be an infinitesimal version of Question $W^{*\Omega}_{\HH^3}$ near the simplest instance of that statement, namely, when $\Omega$ is a totally geodesic plane in $\HH^3$. We will see in Section \ref{sc:minkowski} how Question $W^{\Omega}_{\R^{2,1}}$ is related to Question $W^{\Omega}_{\HH^3}$ and to Question $W^{\Omega}_{\AdS^3}$, see Section \ref{ssc:deformation}. 

\begin{namedquestion}[$W^{\Omega}_{\R^{2,1}}$]
  Let $g_+, g_-$ be two complete metrics of curvature $K\in [-1/\epsilon,-\epsilon]$ on $\DD$, for some $\epsilon>0$, and let $\sigma:\partial \Omega_-\to \partial \Omega_+$. Is there a unique pair of matching convex domains $(\Omega_-, \Omega_+)$ such that the induced metric on $\partial \Omega_\pm$ is isometric to $(\DD, g_\pm)$ and that the identification map between the boundaries of $(\DD,g_-)$ and $(\DD,g_+)$ is $\sigma$? 
\end{namedquestion}

Note that this question is equivalent to Question $W^{*\Omega}_{\HP^3}$ appearing in Section \ref{ssc:quasifuchsian-HP}.

The statement will be explained in more details in Section \ref{sc:minkowski}. When a complete, space-like disk is embedded in a domain of dependence, each ideal point of the disk corresponds to a light-like line in the boundary of the domain. This isotropic line is also contained in the boundary of the matching domain of dependence, and this defines a natural identification between $(\DD,g_-)$ and $(\DD, g_+)$.

The equivariant version of this question -- for corresponding pairs of globally hyperbolic maximal compact Minkowski manifolds -- was recently answered positivey by Graham Smith \cite{smith2020weyl}, while a related polyhedral result was obtained by Fillastre and Prosanov \cite{fillastre-prosanov}.


\section{Background and underlying results}

\subsection{Notations}

We start by fixing some notations used throughout the paper.
\begin{itemize}
\item $\cQS$ is the space of quasi-symmetric, orientation-preserving homeomorphisms from $\RP^1$ to $\RP^1$.
\item $\cT$ is the universal Teichm\"uller space, the quotient of $\cQS$ by the group of M\"obius transformations.
\item $\cML$ is the space of bounded measured laminations on the hyperbolic disk.
\item $S$ denotes a closed, oriented surface of genus at least $2$.
\item $\cT_S$ is the Teichm\"uller space of $S$, that is, the space of complex structures on $S$ considered up to isotopy.
\item $\cML_S$ is the space of measured laminations on $S$.
\item $\cCP_S$ is the space of complex projective structures on $S$, considered up to isotopy.
\end{itemize}

\subsection{The duality between hyperbolic and de Sitter space}
\label{ssc:duality-hyperbolic}

The ``duality'' between the questions concerning the induced metrics and those concerning the third fundamental form is related to the polar duality between the hyperbolic and de Sitter space, which we briefly recall here. A similar polar duality also occurs in anti-de Sitter geometry, see Section \ref{sssc:duality}. We briefly recall this duality here for completeness, referring the reader to e.g. \cite{HR,shu,fillastre-seppi:spherical}.

The de Sitter space can be defined as a quadric in the 4-dimensional Minkowski space:
$$ \dS^3 = \{ x\in \R^{4,1}~|~ \langle x,x\rangle =1 \}~, $$
equipped with the induced metric. It is simply connected, geodesically complete Lorentzian space of constant curvature $1$.

The geodesics (resp. totally geodesic planes) in $\dS^3$ are its intersections with the 2-dimensional planes (resp. hyperplanes) of $\R^{3,1}$ containing $0$. The space-like planes are isometric to the standard sphere, while the time-like planes are topologically cylinders.

The de Sitter space has a projective model, similar to the Klein model of hyperbolic space. In fact only one ``hemisphere'' of $\dS^3$ -- the future of a totally geodesic space-like plane -- is represented as the exterior of the unit ball in $\R^3$, while the interior of this ball contains a projective model of $\HH^3$. This model can also be considered in $\RP^3$, while a projective model of the whole of $\dS^3$ can be obtained in the sphere $S^3$, considered as the double cover of $\RP^3$, as the complement of the union of two balls. 

The ``polar'' duality is defined between points in $\HH^3$ and totally geodesic space-like planes in $\dS^3$, or between points in $\dS^3$ and oriented totally geodesic planes in $\HH^3$. Given a point $x\in \HH^3$, we denote by $x^\perp$ the orthogonal hyperplane in $\R^{3,1}$, which is space-like since $\langle x,x\rangle=-1$ by definition of $\HH^3$. As a consequence, the intersection $x^\perp \cap \dS^3$ is a totally geodesic space-like plane, which we denote by $x^*$. Similarly, given a point $y\in \dS^3$, its orthogonal in $\R^{3,1}$ is an oriented hyperplane, which is time-like since $\langle y,y\rangle=1$. The dual $y^*$ of $y$ is then the intersection of $y^\perp$ with $\HH^3$, considered as a totally geodesic plane. The same definitions can be used to define the dual of a space-like plane in $\dS^3$, which is a point in $\HH^3$, and the dual of an oriented plane in $\HH^3$, which is a point in $\dS^3$.

Consider now a convex subset $\Omega\subset \HH^3$. Its polar dual is the subset $\Omega^*\subset \dS^3$ defined as the set of points dual to the oriented planes in $\HH^3$ bounding a half-plane disjoint from $\Omega$. It is a geodesically convex subset of $\dS^3$. The boundary of $\Omega^*$ is the dual of $\partial \Omega$, in the sense that it is the set of points dual to the oriented support planes of $\Omega$. If $\partial \Omega$ is smooth and strongly convex (meaning that its second fundamental form is positive definite) then $\partial \Omega^*$ is also smooth and strictly convex. 

This definition of polar duality extends to unbounded convex subsets. One can check using the definitions that the dual of a convex subset of $\HH^3$ with ideal boundary a quasicircle $C\subset \partial_\infty \HH^3$ is a convex subset of $\dS^3$ with ideal boundary the same curve $C$. (Note that the ideal boundary of $\HH^3$ is naturally identified with one connected component of the ideal boundary of $\dS^3$, for instance through the projective models described above.) However, the dual of a convex subset with ideal boundary a topological disks $D\subset \partial_\infty \HH^3$ is a convex subset in $\dS^3$ whose ideal boundary is the {\em complement} of $D$.

A key feature of this polar duality (see e.g. \cite{shu,fillastre-seppi:spherical}) is that if $\Sigma \subset \HH^3$ is a strongly convex surface in $\HH^3$, then the dual surface $\Sigma^*\subset \dS^3$ is also strongly convex, with induced metric corresponding to the third fundamental form of $\Sigma$, and third fundamental form corresponding to the induced metric of $\Sigma$. As a consequence, Question $W^{*\Omega}_{\HH^3}$ can be considered as a direct analog of Question $W^{\Omega}_{\HH^3}$ but in the de Sitter space rather than in $\HH^3$.



\subsection{The third fundamental form and the bending lamination of pleated surfaces}
\label{ssc:bending}

Another point that is worth clarifying is the relation between the third fundamental form of smooth surfaces and the measured bending lamination of pleated surfaces, since we claim in the introduction that the questions concerning the bending lamination forms are limit cases of questions concerning the third fundamental form.

It is helpful here to use the polar duality between $\HH^3$ and $\dS^3$. Let $\Sigma$ be a locally convex pleated surface in $\HH^3$, with measured bending lamination $l$. Then the dual of $\Sigma$ is a real tree $\Sigma^*$ in $\dS^3$, with one vertex corresponding to each totally geodesic region in $\Sigma$, see \cite{belraouti1,belraouti2}. Moreover the distance between two vertices in this real tree is equal to the transverse measure for the bending lamination of a transverse segment connecting a point of one of the totally geodesic region to a point of the other.

Now if $\Sigma_r$ is the equidistant surface at constant distance $r>0$ from $\Sigma$ on the concave side of $\Sigma$, then $\Sigma_r^*$ converges as $r\to 0$ to a convex surfaces in $\dS^3$ which contains $\Sigma^*$ and is light-like outside of $\Sigma^*$. Moreover, $\Sigma_r$ is a $C^{1,1}$ surface, and its third fundamental form is well-defined. It is then proved in \cite{belraouti1,belraouti2} that the distance for the third fundamental form on $\Sigma_r$ converges as $r\to 0$ to the pseudo-distance determined by the transverse measure of $l$. More specifically, if $x,y\subset \Sigma$ are contained in totally geodesic open subsets, if $x_r, y_r\subset \Sigma_r$ and $x_r\to x, y_r\to y$ as $r\to 0$, then
$$ d_{\III_r}(x_r, y_r)\to d_l(x,y)~, $$
where $d_{\III_r}$ is the distance associated to the third fundamental form on $\Sigma_r$, while $d_l$ is the pseudo-distance associated to the measured lamination $l$ on $\Sigma$. We refer to \cite{belraouti1,belraouti2} for details. 



\section{Conformal structures on the boundaries of convex subsets}

\subsection{Induced metric conformal structures}
\label{ssc:conformalI}

We provide here a proof of Lemma \ref{lm:existence-confI}. In the next section we state and prove a similar lemma for the third fundamental forms.

\begin{proof}[Proof of Lemma \ref{lm:existence-confI}]
  (1) Assume that $\partial\Omega$ has principal curvatures at most $k_0$. We consider the hyperbolic Gauss map $G:\partial \Omega\to \partial_\infty\HH^3$, sending a point $x\in \partial\Omega$ to the endpoint at infinity of the geodesic ray starting from $x$ orthogonally to $\partial\Omega$. The convexity of $\Omega$ indicates that $G$ is a homeomorphism from $\partial_0\Omega$ to $\partial_\infty\HH^3$. The bound on the principal curvatures implies that $G$ is $K$-quasiconformal, with $K=\sqrt{k_0^2+2k_0}/(1+k_0)$.

  Let $\mu$ be Beltrami coefficient of the inverse map $G$, considered on $\partial_\infty\HH^3\simeq \CP^1$. The measurable Riemann mapping theorem \cite{ahlfors-bers:riemann} implies that there is quasiconformal map $H:\CP^1\to \CP^1$ with Beltrami coefficient  $-\mu$, and it is unique up to composition on the left by a M\"obius transformation. The composition $H\to G:\partial_0\Omega\to \CP^1$ is then conformal, which proves the existence of a induced metric conformal map for $\Omega$. Its uniqueness (up to post-composition by M\"obius transformations) follows from the uniqueness of $H$.

  (2) The result will follow from \cite[Lemma 4.1]{weylgen}, applied to each connected component of $\partial\Omega$, but we need to prove that the induced metrics on those connected components have injectivity radius bounded from below by a positive constant.

  Let $\partial_i\Omega$ be a connected component of $\partial\Omega$, and let $\gamma$ be a closed geodesic on $(\partial_i\Omega, I)$ of length $l$ -- we will show that $l$ cannot be smaller than a fixed, strictly positive constant.

  Since $\partial_i\Omega$ is a topological disk, $\gamma$ bounds a topological disk $D_\gamma$. Since $\gamma$ is geodesic for $I$, the integral curvature of $D_\gamma$ is equal to $2\pi$. The area of $D_\gamma$ is therefore at least $2\pi\epsilon'$.

  However the hypothesis that $I$ has bounded derivatives implies (for $k=0$) that $I=e^{2u}h_{-1}$, where $h_{-1}$ is the complete hyperbolic metric conformal to $I$ on $\partial_i\Omega$, with $-C_0\leq u\leq C_0$. The area of $D_\gamma$ for $h_{-1}$ is therefore at least $2\pi \epsilon'e^{-2C_0}$. This in turns leads, through the hyperbolic isoperimetric inequality, to a lower bound on the length of $\gamma$ for $h_{-1}$, and then, using this time the upper bound $I\leq e^{2C_0}h_{-1}$, to a lower bound on the length of $\gamma$ for $I$.

  The lower bound on the injectivity radius of $(\partial_i\Omega,I)$ follows, and point (2) then follows from \cite[Lemma 4.1]{weylgen}.
\end{proof}

It can be noted that, if the induced metric is assumed to have non-positive sectional curvature, then the lower bound on the injectivity radius is trivial. However in the presence of positive curvature, there might be complete metrics with curvature bounded from above and from below but arbitrarily small injectivity radius.

\subsection{Third fundamental form conformal structures}
\label{ssc:conformalIII}

There is a similar lemma but for the conformal structures associated to the third fundamental form on $\partial \Omega$, or rather for the induced metric on the dual convex subset, as seen in Definition \ref{df:III}. Specifically, the following lemma is a direct analog of Lemma \ref{lm:existence-confI}.

\begin{lemma} \label{lm:existence-confIII}
  \begin{enumerate}
  \item Assume that there exists $k_0>0$ such that $\partial\Omega$ has principal curvatures at most $k_0$. Then $\partial_0\Omega$ admits a unique third fundamental form conformal structure.
  \item This is in particular the case if there exist $\epsilon'>0$ and $k>0$ such that 
    \begin{itemize}
    \item $\partial\Omega$ is a disjoint union of topological disks,
    \item the third fundamental form on $\partial\Omega$ has curvature $K^*\in [-1/\epsilon', 1-\epsilon']$ and bounded derivatives,
      \item closed geodesics of $(\partial \Omega, \III)$ have length at least $2\pi+\epsilon'$.
    \end{itemize}
  \end{enumerate}
\end{lemma}

\begin{proof}[Sketch of the proof]
  Point (1) can be proved exactly as point (1) of Lemma \ref{lm:existence-confIII}. Point (2) however requires a bit of adaptation. The proof essentially follows the proof of  \cite[Lemma 4.1]{weylgen}, however the use of \cite[Th\'eor\`eme D]{L1} is replaced by that of \cite[Th\'eor\`eme 5.6]{these}, which is an analog in the de Sitter space -- the hypothesis on the lengths of closed geodesics in $(\partial \Omega, \III)$ then prevents the appearance of a singularity along a geodesic in the limit, since closed geodesics on the de Sitter space have length $2\pi$. 
\end{proof}

\subsection{Mixed conformal structures}

Nothing prevents one from considering ``mixed'' conformal structures, taking the conformal class of the induced metrics on some boundary components of $\partial \Omega$, and the third fundamental form on others. Such boundary data would occur for instance when considering a convex subset $\Omega$ which is the universal cover of a convex subset $N$ in a convex co-compact hyperbolic manifolds $M$, when one attempts to prescribe the induced metric on some boundary components of $N$ and the third fundamental form on others. This type of mixed problems, for compact subsets of convex co-compact hyperbolic manifolds,  is considered in \cite{chen2022geometric,mesbah:induced}.  

\section{Convex domains in hyperbolic space}
\label{sc:Omega}

We develop in this section the questions and results already mentioned more briefly in Section \ref{ssc:intro-domain}, and give more precise statements of both known results and specific questions that appear natural in light of the general questions stated in the introduction.

We first focus on convex domains with ideal boundary a quasi-circle, and then explain the relations with geodesically convex subsets in quasifuchsian, or more generally geometrically finite, hyperbolic 3-manifolds.

\subsection{Convex domains with ideal boundary a quasi-circle}
\label{ssc:thin-universal}

The simplest case to consider, topologically speaking, is a convex subset $\Omega\subset \HH^3$ such that the ideal boundary $\partial_\infty \Omega$ is just one quasicircle. Then $\partial\Omega$ -- the boundary of $\Omega$ -- is the disjoint union of two complete surfaces $\partial_-\Omega$ and $\partial_+\Omega$, each diffeomorphic to the disk $\DD$ and having $\partial_\infty \Omega$ as their common asymptotic boundary.

Both $\partial_-\Omega$ and $\partial_+\Omega$, equipped with its induced metric, is then conformal to $\DD$, the unit disk in $\R^2$, so that they can both be equipped with a conformal boundary $\partial_c(\partial_-\Omega)$ and $\partial_c(\partial_+\Omega)$, both equipped with a real projective structure. Both $\partial_c(\partial_\pm\Omega)$ are identified homeomorphically to $\partial_\infty\Omega$, and those identifications define a map $\sigma:\partial_c(\partial_-\Omega) \to \partial_c(\partial_+\Omega)$

\begin{question}
  Is this map $\sigma$ $k$-quasi-symmetric, for a constant $k$ depending only on the quasi-symmetric constant of $\partial_\infty\Omega$?
\end{question}

We can now ask whether Question $W^{\Omega}_{\HH^3}$ has a positive answer in this restricted setting.

\begin{question} \label{q:1}
  Let $h_-, h_+$ be two complete, conformal metrics of curvature $K\in [-1+\epsilon,-\epsilon]$ on the disk $\DD$, and let $\sigma:\partial \DD\to \partial \DD$ be an orientation-reversing quasi-symmetric homeomorphism. Is there a unique convex subset $\Omega\subset \R^3$, with $\partial_\infty\Omega$ a quasi-circle, such that the induced metric on $\partial_\pm\Omega$ is $h_\pm$, with gluing map at infinity corresponding to $\sigma$? 
\end{question}

The {\em existence} is proved in \cite{convexhull} when $h_-$ and $h_+$ have constant curvature $K\in [-1,0)$. This existence result therefore includes the case of pleated surfaces, corresponding to $K=-1$.

For variable curvature metrics, a result is obtained in \cite[]{weylgen}, under somewhat stronger (and probably not necessary) regularity assumptions.

\begin{definition} \label{df:bounded-der}
  Let $h$ be a complete metric on $\DD$, which can be written as $e^{2u}h_0$, where $h_0$ is the conformal hyperbolic metric on $\DD$. We will say that $h$ has {\em bounded derivatives} if for all $n\in \N$ there exists $c_n>0$ such that the norm of the $n$th derivatives of $u$ for $h_0$ is everywhere at most $c_n$.
\end{definition}

\begin{theorem}
  The answer to Question \ref{q:1} is positive if $h_-$ and $h_+$ are supposed to have bounded derivatives.
\end{theorem}

Note that the upper bound on the curvature in Question \ref{q:1} could be taken to be larger than $0$, we voluntarily state the question in a relatively restricted setting. 

The uniqueness however remains elusive.

Similarly, Question $W^{*\Omega}_{\HH^3}$ can be stated, for convex domains with boundary at infinity a quasicircle, as follows.

\begin{question} \label{q:1*}
  Let $h_-, h_+$ be two complete, conformal metrics of curvature $K\in [-1/\epsilon,1-\epsilon]$ on the disk $\DD$, with closed geodesics of length larger than $2\pi$, and let $\sigma:\partial \DD\to \partial \DD$ be an orientation-reversing quasi-symmetric homeomorphism. Is there a unique convex subset $\Omega\subset \R^3$, with $\partial_\infty\Omega$ a quasi-circle, such that the third fundamental form on $\partial_\pm\Omega$ is $h_\pm$, with gluing map at infinity corresponding to $\sigma$? 
\end{question}

Again, the existence part is shown in \cite{convexhull} in the special case of metrics of constant curvature $K\in (-1,0)$. However, this result does not include the case of pleated surfaces, for which one would like to prescribe the bending measured laminations on the two connected components of the boundary of the convex hull of a quasicircle in $\partial_\infty\HH^3$. A possible statement is suggested by \cite[Theorem B]{laminations}, which gives an existence result for the corresponding problem in anti-de Sitter space, see Section \ref{sc:ads}. We will say that two measured laminations $\lambda_-,\lambda_+$ on the hyperbolic disk $\DD$ {\em strongly fill} if, for every $\epsilon>0$, there exists $c>0$ such that for any geodesic segment $\gamma$ in $\DD$ of hyperbolic length at least $c$, $i(\lambda_-,\gamma)+i(\lambda_+,\gamma)\geq \epsilon$. We also need to remember that a measured lamination on $\DD$ can be considered as a measure on $\partial \DD\times \partial \DD\setminus \Delta$, where $\Delta\subset \partial \DD\times \partial \DD$ is the diagonal, and notice that given a quasicircle $C\subset \partial_\infty \HH^3$, the measured bending laminations on the two connected components of boundary the convex hull of $C$ can be considered as measures on $C\times C\setminus \Delta$, where again $\Delta\subset C\times C$ is the diagonal.

\begin{question} \label{q:1*+}
  Let $\lambda_-, \lambda_+$ be two bounded measured laminations on the hyperbolic disk $\DD$ that strongly fill, without leaf of weight at least $\pi$. Is there a unique parameterized quasicircle $u:\partial \DD^1\to \partial_\infty \HH^3$ such that $u_*(\lambda_-)$ and $u_*(\lambda_+)$ are the measured bending laminations on the connected components of the boundary of the convex hull of $u(\partial \DD)$? 
\end{question}

Here a {\em parameterized quasi-circle} $u:\partial \DD\to \partial_\infty \HH^3$ is a continuous injective map such that if $v_\pm:\DD\to \partial_\infty\HH^3$ are the Riemann uniformization maps of the two connected component of $\partial_\infty\HH^3\setminus u(\partial \DD)$, then $v_\pm^{-1}\circ v:\partial \DD\to \partial \DD$ are both quasi-symmetric.

It appears possible that the type of arguments used in \cite{laminations} could lead to a proof of the existence part of Question \ref{q:1*+}, but a serious technical question (concerning compactness) needs to be resolved.

\subsection{Bounded convex domains in quasifuchsian manifolds}
\label{ssc:quasifuchsian}

The questions in Section \ref{ssc:thin-universal} take a somewhat simpler form when considered for quasicircles that are invariant under a quasifuchsian group action. Let $\Omega\subset \HH^3$ be a convex subset invariant under such an action $\rho:\Gamma\to PSL(2,\C)$, where $\Gamma$ is the fundamental group of a closed surface of genus at least $2$. The quotient of $\Omega$ by $\rho(\Gamma)$ is then a geodesically convex subset in $M=\HH^3/\rho(\Gamma)$, which by our hypothesis on $\rho$ is a {\em quasifuchsian} hyperbolic manifold.

Recall that a quasifuchsian manifold is a complete 3-dimensional hyperbolic manifold, homeomorphic to the product by an interval of a surface of genus at least $2$, and containing a non-empty, compact geodesically convex subset. A special case is constitued by {\em Fuchsian} manifolds, which contain a closed, totally geodesic surface. If $M$ is a non-Fuchsian quasifuchsian manifold, then any non-empty geodesically convex subset in $M$ is a 3-dimensional domain with boundary the disjoint union of two connected surfaces. 

Question \ref{q:1} corresponds in this restricted context to the following statement, proved in \cite{L4} (for the existence part) and \cite{hmcb} (for the uniqueness): given a closed surface $S$ of genus at least $2$ and two smooth Riemannian metrics $h_-, h_+$ of curvature $K>-1$ on $S$, there exists a unique quasifuchsian manifold $M$ homeomorphic to $S\times \R$ and a unique geodesically convex subset $\Omega\subset M$ such that the induced metrics on the two boundary components of $\Omega$ are isotopic to $h_-$ and $h_+$. Similarly, \cite{hmcb} contains an answer to Question \ref{q:1*} in the quasifuchsian setting: given two smooth metrics $h_-, h_+$ on $S$ of curvature $K<1$ with closed, contractible geodesics of length $L>2\pi$, there exists a unique quasifuchsian manifold and a unique geodesically convex subset $\Omega\subset M$ such that the third fundamental forms on the two boundary components of $\Omega$ are isotopic to $h_-$ and $h_+$.

However the uniqueness part of the same question concerning the induced metric does not extend to domains with boundary a union of convex pleated surfaces. The only closed convex pleated surfaces in a quasifuchsian manifold are the boundary components of the convex core, the smallest non-empty geodesically convex subset that they contain. Questions $W^{\Omega}_{\HH^3}$ therefore take, in this quite specific setting, the following form. 

\begin{conjecture}[Thurston] \label{cj:thurston-m}
  Let $m_-, m_+$ be two hyperbolic metrics on a closed surface $S$ of genus at least $2$. Is there a unique quasifuchsian hyperbolic structure on $S\times \R$ such that the induced metrics on the boundary components of the convex core are isotopic to $m_-$ and $m_+$? 
\end{conjecture}

The existence part of this conjecture is known, it is a consequence of results of Epstein-Marden \cite{epstein-marden} or Labourie \cite{L4}.

Question $W^{*\Omega}_{\HH^3}$ takes, in this special case of pleated surfaces in quasifuchsian manifolds, the form of another conjecture of Thurston: non-Fuchsian quasifuchsian manifolds are uniquely determined by the measured pleating lamination on the boundary of their convex core. The existence part of this question was settled by Bonahon and Otal \cite{bonahon-otal}, who completely characterized the possible measured bending lamination, while the uniqueness was recently proved in \cite{uniqueness}.

\subsection{Bounded convex domains in convex co-compact manifolds}
\label{ssc:cocompact}

A {\em convex co-compact} hyperbolic manifold is a complete, 3-dimensional hyperbolic manifold which contains a non-empty, compact, geodesically convex subset. Quasifuchsian manifolds are among the simplest examples. 

Most of the results quoted above for quasifuchsian manifolds extend to convex co-compact hyperbolic manifolds. 

Other results can be found in \cite{ideal,hphm} for geodesically convex domains in quasifuchsian manifolds which are locally like an ideal or hyperideal hyperbolic polyhedron. More recently, Prosanov has obtained several results which are directly related to Questions $W^{\Omega}_{\HH^3}$ and $W^{*\Omega}_{\HH^3}$. In \cite{prosanov:dual}, he considers convex domains in convex co-compact hyperbolic manifolds which are ``polyhedral'', in the sense that their boundaries are locally isometric to that of compact polyhedra in $\HH^3$, and he shows that those convex domains are uniquely determined by their {\em dual metrics} -- the dual metric is the analog, in this polyhedral setting, of the third fundamental form. In \cite{prosanov:polyhedral}, he considers the same class of convex polyhedral domains in convex co-compact hyperbolic manifolds, but focuses on the induced metrics on their boundaries. This leads him to introduce a wider class of convex subsets in convex co-compact manifolds, which he calls ``bent'', and he shows a broad realization result for those induced metrics. He also exhibits a large class of boundary metrics which are realized on polyhedral (rather than bent) domains.

Questions $W^{\Omega}_{\HH^3}$ and $W^{*\Omega}_{\HH^3}$ cover those cases, in the sense that a positive answer to those questions would provide a new proof of the existence and uniqueness of convex co-compact manifold containing a geodesically convex subset with prescribed induced metric (or third fundamental form) on the boundary. It would be sufficient to have a positive answer to Questions $W^{\Omega}_{\HH^3}$ and $W^{*\Omega}_{\HH^3}$ under the additional asumption that the boundary of each connected component of $U$ is a quasi-disk. In fact the following question is a very special case of Question $W^{\Omega}_{\HH^3}$.

\begin{question}
  Let $\Gamma$ be a Jordan curve in $\CP^1$, and let $h_-$ and $h_+$ be complete metrics of constant curvature $K\in [-1,0)$ on the two connected component of $\CP^1\setminus \Gamma$. Is there a unique convex subset $\Omega\subset \HH^3$ such that there is a conformal homeomorphism from $\CP^1$ to $\partial_0\Omega$ (with the notations of Question $W^{\Omega}_{\HH^3}$) which sends $\Gamma$ to $\partial_\infty\Omega$ and is an isometry on $\CP^1\setminus \Gamma$? 
\end{question}

This question can also be stated in terms of the induced metric conformal structures see in Definition \ref{df:I}.

It is proved in \cite{convexhull} that the existence part of this statement holds when $\Gamma$ is a quasicircle. However it is not clear whether this condition ($\Gamma$ is a quasicircle) is necessary, whether for existence or for uniqueness.

\subsection{Fuchsian results}
\label{ssc:fuchsian}

Special cases of the questions considered in the previous section are known to hold for surfaces that are invariant under a Fuchsian group action --- that is, a group action that leaves invariant a plane. This happens when the two metrics $h_-$ and $h_+$ considered above are identical. Existence and uniqueness results of this type were obtained for the third fundamental form of smooth surfaces in $\HH^3$ in \cite{iie}, and for the induced metrics and dihedral angles on polyhedral surfaces in e.g. \cite{fillastre2,fillastre3,fillastre4,leibon1,prosanov:ideal} and even for general convex surfaces in Fuchsian manifolds \cite{prosanov:fuchsian}.

Another relatively simple setting is that of surfaces invariant under a parabolic group action. This corresponds geometrically to the universal cover of a surface embedded in a ``cusp'', a hyperbolic manifold homeomorphic to $T^2\times \R$, where $T^2$ is the torus. See for instance \cite{fillastre-izmestiev}.

In view of those results for surfaces invariant under a Fuchsian group action --- and therefore with an ideal boundary that is a round circle --- it seems natural to ask whether Question $W^{\Omega}_{\HH^3}$ has a positive answer assuming only that the ideal boundary is a circle (but without assuming invariance under a group action).

\begin{question}
  Let $h$ be a complete conformal metric on the disk $\DD$, with curvature $K>-1$. Is there a unique isometric embedding of $(\DD, h)$ in $\HH^3$ such that the ideal boundary of the image is a round circle?
\end{question}

We can also ask the dual question.

\begin{question}
  Let $h$ be a complete conformal metric on the disk $\DD$, with curvature $K<1$ and closed geodesics of length $L>2\pi$. Is there a unique embedding of $(\DD, h)$ in $\HH^3$ with third fundamental form $h$ such that the ideal boundary of the image is a round circle?
\end{question}

In addition, those two questions can also be considered for ``parabolic'' embeddings, that is, when the image surface has only one point at infinity.

\begin{question}
  Let $h$ be a complete hyperbolic metric on the disk $\DD$, with curvature $K>-1$. Assume that $(\DD, h)$ is conformal to $\C$. Is there a unique isometric embedding of $(\DD, h)$ in $\HH^3$ such that the ideal boundary of the image is a point?
\end{question}

We can also ask the dual question.

\begin{question}
  Let $h$ be a complete hyperbolic metric on the disk $\DD$, with curvature $K<1$ and closed geodesics of length $L>2\pi$. Assume that $(\DD, h)$ is conformal to $\C$. Is there a unique embedding of $(\DD, h)$ in $\HH^3$ with third fundamental form $h$ such that the ideal boundary of the image is a point?
\end{question}

From a technical point of view, the most interesting feature of those questions is that the infinitesimal rigidity, which is arguably the most challenging aspect of the questions considered here, might be amenable to relatively simple geometric arguments.





\section{Immersed locally convex surfaces}
\label{sc:immersions}

We now turn to Question $W^{imm}_{\HH^3}$ and its dual, Question  $W^{*imm}_{\HH^3}$. We focus here mostly on those questions when restricted to locally convex pleated surfaces, possibly with cusps, where they correspond to either relatively recent results or to well-known open questions. Considering more general cases of Questions $W^{imm}_{\HH^3}$ and $W^{*imm}_{\HH^3}$ leads to new questions on the grafting map, and on new perspectives on questions concerning circle packings or circle patterns.

\subsection{Pleated surfaces and grafting}

\subsubsection{Immersed pleated surfaces in $\HH^3$}

Perhaps the simplest instance of Question $W^{imm}_{\HH^3}$ is when $h$ is hyperbolic, that is, of constant curvature $-1$. The image of $\phi$ then needs to be a locally convex, pleated surface in $\HH^3$. The question can then be formulated in the following simpler form.

\begin{question} \label{q:grafting-univ}
  Let $\sigma:\RP^1\to \RP^1$ be a quasi-symmetric homeomorphism. Is there a unique locally convex pleated isometric immersion $\phi:\HH^2\to \HH^3$ such that the induced metric gluing map of $\phi$ is $\sigma$?
\end{question}

Here the gluing map that is considered is between the induced metric and the pull-back by the Gauss map of the conformal class at infinity, see Section \ref{ssc:complete}.

Question $W^{*imm}_{\HH^3}$ takes for pleated surfaces a slightly different from, since the notion of gluing map has to be formulated a bit differently between a measured lamination and a conformal disk. Recall that given a locally convex immersed surface $S\subset \HH^3$, the hyperbolic Gauss map $G:S\to \partial_\infty \HH^3$ sends a point $x\in S$ to the endpoint at infinity of the geodesic ray starting from $x$ orthogonally to $S$.

This definition extends to locally convex pleated surfaces, but the hyperbolic Gauss map is then defined on the unit normal bundle $N^1S$ of $S$ (the space of unit vectors normal to oriented support planes of $S$).

\begin{question} \label{q:grafting-univ*}
  Let $l$ be a bounded measured lamination on the disk $\DD$. Is there a unique immersed locally convex pleated surface $\Sigma\subset \HH^3$ together with a map $u:\DD\to N^1\Sigma$ which is a parameterization of $N^1\Sigma$ such that
  \begin{itemize}
  \item $G\circ u$ is conformal, where $G$ is the hyperbolic Gauss map of $\Sigma$,
  \item the pull-back by $u$ of the measured bending lamination of $\Sigma$ corresponds to $l$?
  \end{itemize}
\end{question}

Note that the pull-back by $u$ of the measured bending lamination of $S$ is naturally a measured {\em foliation} on $\DD$, and the second point requests that this measured foliation corresponds to $l$ under the natural correspondence between measured foliations and measured laminations on $\DD$. Prescribing the measured foliation on the disk at infinity is similar, heuristically at least, to prescribing the ``gluing'' between the measured pleating lamination and the conformal structure at infinity. (Note also that Question \ref{q:grafting-univ*} is stated in a slightly indirect way, because $G$ is not necessarily injective.)


\subsubsection{Pleated surfaces in hyperbolic ends}

Questions \ref{q:grafting-univ} and \ref{q:grafting-univ*} can be considered in the special case of surfaces invariant under a surface group representation $\rho:\pi_1(S)\to PSL(2,\C)$. They have positive answers in those restricted cases, and are in fact equivalent to two questions on the grafting map which have been solved previously.

Recall that, given a closed surface $S$ of genus at least $2$, the grafting map $gr:\cT_S\times \cML_S\to \cT_S$ is such that, if a quasifuchsian manifold $M$ (or more generally a hyperbolic end $E$) contains a locally convex pleated surface $S$ with induced metric $m$ and measured bending lamination $l$, then the conformal structure at infinity on the ideal boundary component facing $S$ (on the concave side of $S$) is $gr(m,l)$, see e.g. \cite{dumas-survey}.

In this case, the gluing map $\sigma:\RP^1\to \RP^1$ in Question \ref{q:grafting-univ} is equivariant under a pair or representations $\rho_1,\rho_2:\pi_S(S)\to PSL(2,\R)$. Then $\HH^2/\rho_1(\pi_S(S))$ is the quotient of the pleated surface in $\HH^3$, while $\HH^2/\rho_2(\pi_1(S))$ is the hyperbolic metric in the pull-back by the hyperbolic Gauss map of the conformal class at infinity of $\HH^3$. This conformal class is obtained by {\em grafting} the hyperbolic metric on the pleated surface along its measured bending lamination. 

Question \ref{q:grafting-univ}, in this restricted context, is therefore equivalent to the following result of Dumas and Wolf \cite{dumas-wolf}.

\begin{theorem}[Dumas--Wolf] \label{tm:dumas-wolf}
  Let $h$ be a hyperbolic metric on a closed surface $S$. The grafting map $gr(h,\cdot):\cML_S\to \cT_S$ is a homeomorphism.
\end{theorem}

Still for pleated surfaces invariant under a surface group representation in $PSL(2,\C)$, the measured bending lamination considered above is the lift to the universal cover of a measured bending lamination $l$ on $S$, while the condition in Question \ref{q:grafting-univ*} is equivalent to asking that the pull-back on $S$ by the hyperbolic Gauss map of the conformal structure at infinity is prescribed. A positive anwer to Question \ref{q:grafting-univ*} in this equivariant case therefore follows from a result of Scannell and Wolf \cite{scannell-wolf}.

\begin{theorem}[Scannell-Wolf] \label{tm:scannell-wolf}
  Let $l$ be a measured lamination on a closed surface $S$ of genus at least $2$. The grafting map $gr(\cdot, l):\cT_S\to \cT_S$ is an analytic diffeomorphism.
\end{theorem}

\subsubsection{Grafting and pleated disks in $\HH^3$}

Questions \ref{q:grafting-univ} and \ref{q:grafting-univ*} can be formulated in terms of grafting on the hyperbolic disk. This is particularly natural in view of Thurston's Earthquake Theorem, which states that any quasi-symmetric homeomorphism of $\RP^1$ is the boundary map of a unique left earthquake on the hyperbolic plane (see Section \ref{sssc:earthquakes}). Indeed, earthquakes are intimately related to grafting (see e.g. \cite{mcmullen:complex}) while we will see below that Thurston's Earthquake Theorem gives a positive answer to question $W^{imm}_{AdS}$ in the special case of locally convex pleated disks in $AdS^3$. 

To state the questions in a simple manner, we introduce notations for the grafting map on the hyperbolic disk. Let $l$ be a bounded measured lamination on $\DD$. Let $\Sigma\subset \HH^3$ be a pleated disk in $\HH^3$ obtained by pleating the hyperbolic disk along the measured bending lamination $l$, and let $G:N^1\Sigma\to \CP^1$ be the hyperbolic Gauss map. We also denote by $u:\DD\to N^1\Sigma$ the Riemann uniformization map for the conformal structure obtained on $N^1\Sigma$ by pull-back of the conformal structure at infinity of $\HH^3$. Finally we denote by $\pi:N^1\Sigma\to \Sigma$ the canonical projection.

\begin{lemma} \label{lm:qconf}
  There exists $\delta>0$ such that, if $l$ is bounded, then the map $\pi\circ u:\DD\to \Sigma$ is within distance $\delta$ from a quasi-conformal map, with quasi-conformal constant depending only on the bound on $l$.
\end{lemma}

\begin{proof}
  The proof can be done following the arguments in \cite{epstein-marden}, in a closely related context. 
\end{proof}

As a consequence of Lemma \ref{lm:qconf}, the map $\pi\circ u:\DD\to \Sigma$ extends so a quasi-symmetric map from $\RP^1$ to $\partial_\infty\Sigma$. 

Moreover, $l$ lifts to a measured foliation on $N^1\Sigma$. For instance, if $l$ contains a leaf $\cL$ with atomic weight $w$, then the inverse image of $\cL$ by $\pi$ is a strip in $N^1\Sigma$, which is foliated by parallel lines, with total transverse weight $w$. This measured foliation can be pulled back to $\DD$ by $u$, yielding a measured foliation on $\DD$, which corresponds to a unique measured lamination on $\DD$. 

\begin{definition}
  We denote by:
  \begin{itemize}
  \item $GR_0(l):\RP^1\to \partial_\infty\Sigma$ the boundary extension of $\pi\circ u$,
  \item $GR_1(l)$ the measured lamination corresponding to $(\pi\circ u)^*(l)$ on $\DD$,
  \item $GR(l)=(G_0(l), G_1(l))$.
  \end{itemize}
\end{definition}

Note that it is also possible to define a ``universal grafting map'', which associates to a measured lamination on $\DD$ a complex projective structure also on $\DD$. This map however does not appear to be directly related to the topics considered here.

It follows quite directly from the definition that Question \ref{q:grafting-univ} can be stated equivalently as follows.

\begin{question} \label{q:grafting-univ2}
  Is the map $GR_0:\cML\to \cT$ a homeomorphism from bounded measured laminations to quasi-symmetric homeomorphisms?
\end{question}

Moreover, Question \ref{q:grafting-univ*} can be reformulated equivalently in terms of grafting, too.

\begin{question} \label{q:grafting-univ2*}
  Is the map $GR_1:\cML\to \cML$ a homeomorphism from bounded measured laminations to bounded measured laminations? 
\end{question}




The equivalence between Question \ref{q:grafting-univ*} and Question \ref{q:grafting-univ2*} should be quite clear from the definitions.

\subsection{Circle patterns and circle packings}

This section is focused on a slightly different type of pleated surfaces in $\HH^3$ and in hyperbolic manifolds: those which have cusps, and therefore look locally like ideal hyperbolic polyhedra. Those surfaces are closely related to circle packings and to Delaunay circle patterns (see below). Through this correspondence, an interesting question on circle patterns turns out to be a special case of Question $W^{*imm}_{\HH^3}$. 

\subsubsection{Convex surfaces in hyperbolic ends and circle packings}
\label{sssc:ends}

One motivation for Question $W^{*imm}_{\HH^3}$ is a beautiful conjecture of Kojima, Mizushima and Tan \cite{KMT,KMT2,KMT3}. They considered circle packings on closed surfaces equipped with a complex projective structure, which are more flexible than circle packings on closed hyperbolic or Euclidean surfaces. Kojima, Mizushima and Tan consider a graph $\Gamma$ that is the 1-skeleton of a triangulation of a closed surface $S$, and the space $\cC_\Gamma$ of complex projective structures on $S$ admitting a circle packing with incidence graph $\Gamma$.

\begin{conjecture}[Kojima, Mizushima and Tan] \label{cj:kmt}
  The forgetful map from $\cC_\Gamma$  to the Teichm\"uller space of $S$ is a homeomorphism.
\end{conjecture}

The Koebe circle packing theorem is the special case when $S$ is a sphere, since there is then a unique complex projective structure at infinity. Kojima, Mizushima and Tan verified the conjecture in some other simple cases, in particular on the torus for a circle packing with only one circle.

In \cite{delaunay} we noticed that Conjecture \ref{cj:kmt} can be extended from circle packings to ``Delaunay circle patterns'' -- patterns of circles that occur as the Delaunay decomposition of a finite set of points on a surface equipped with a complex projective structure. We also proved a statement that could be ``one half'' of this generalized conjecture, namely, that the projection map to Teichm\"uller space is proper.

A key point in \cite{delaunay} is that a Delaunay circle pattern on a surface with a complex projective structure can be considered as a polyhedral surface (with all vertices at infinity) in a hyperbolic end. More specifically, Conjecture \ref{cj:kmt} is in fact {\em equivalent} to Question $W^{*imm}_{\HH^3}$ considered in the special case of {\em equivariant} convex immersions of surfaces that are locally like ideal polyhedra.

Let us briefly explain how the correspondence between Delaunay circle patterns (or circle packings) and ``ideal'' polyhedral surfaces functions. A Delaunay circle pattern is basically the sort of circle pattern that occurs when considering the Delaunay decomposition of a discrete set of points on a surface equipped with a complex projective structure (for instance, an open subset of $\CP^1$), see \cite[Definition 1.9]{delaunay}. The points in this discrete set will be the vertices of the ideal polyhedral surface. Given such a Delaunay circle pattern, each empty disk (disk containing no vertex) is the boundary at infinity of a unique half-plane. The boundary of the complement of the union of the half-planes corresponding to all the disks is the ideal polyhedral surface. This construction can be done either in a hyperbolic end, or in its universal cover, in which case both the circle pattern and the corresponding ideal polyhedral surface are invariant under a surface group action.

Circles packings (as considered in \cite{KMT}, that is, when faces of the incidence graph are triangles) can be considered as special cases of Delaunay circle patterns if one adds for each face of the dual graph another circle, orthogonal to the adjacent circles. Adding those ``dual'' circles yields a Delaunay circle pattern for which all intersection angles are $\pi/2$.

Question 1.13 in \cite{delaunay} asks whether given a Delaunay circle pattern $\cC$ on a surface $S$ equipped with a complex projective structure $c\in \cCP_S$, $\cC$ and $c$ are uniquely determined by the incidence graph and intersection angles of $\cC$, together with the complex structure underlying $c$. And whether, given an embedded graph $\Gamma$ in $S$ and a set of weights in $(0,\pi)$ associated to the edges of $\Gamma$, satisfying some natural conditions (basically corresponding to the conditions of Question $W^{*imm}_{\HH^3}$), there is for each complex structure $X$ on $S$ a unique complex projective structure $c$ compatible with $X$, admitting a circle pattern with incidence graph $\Gamma$ and intersection angles given by the prescribed weights. A positive answer was recently given by W. Y. Lam for the torus, see \cite{lam2019quadratic,lam2024space}.

\subsubsection{Infinite circle patterns and circle packings}

The point of view of Question $W^{*imm}_{\HH^3}$ suggests an extension of Conjecture \ref{cj:kmt} to infinite circle patterns or infinite circle packings. We state the question here in a relatively limited setting for simplicity.

\begin{definition}
  A Delaunay circle pattern in the hyperbolic disk $\DD^2$ is {\em bounded} if it covers the whole of $\DD$ and the hyperbolic radius of the circles is bounded from above and is bounded from below by a positive constant.
\end{definition}

\begin{question} \label{q:delaunay-infinite}
  Let $\cC$ be a bounded Delaunay circle pattern in the hyperbolic disk $\DD$, and let $\sigma:\RP^1\to \RP^1$ be a quasi-symmetric homeomorphism. Is there a unique local homeomorphism $u:\DD\to \CP^1$ sending each circle to a round circle, preserving the intersection angles, and such that if $v:\DD\to (\DD,u^*c_0)$ is the Riemann uniformization map, then $\partial v:\RP^1\to \RP^1$ is equal (up to M\"obius transformation) to $\sigma$?
\end{question}

Again, this statement is a special case of Question $W^{*imm}_{\HH^3}$. The third fundamental form becomes, for surfaces which are locally isometric to the boundary of an ideal polyhedron, a measure transverse to the edges, with weight the exterior dihedral angle, so prescribing $\III$ in this case means prescribing the combinatorics and exterior dihedral angles of an ideal polyhedral surface, which means prescribing the combinatorics and intersection angles of a Delaunay circle pattern. 

\subsection{Parameterization of locally convex equivariant embeddings}

We believe that a positive answer to Question $W^{imm}_{\HH^3}$, and the dual Question $W^{*imm}_{\HH^3}$, would provide a general description of equivariant convex isometric embeddings of surfaces of higher genus in $\HH^3$.

Let $(S,g)$ be a closed surface of genus at least $2$, equipped with a Riemannian metric of curvature $K>-1$. If we consider the Alexandrov Theorem \ref{tm:alexH} as a statement on locally convex isometric embeddings, it is quite natural to ask for a description of the space of all isometric immersions of the universal cover $(\tilde S, \tilde h)$ in $\HH^3$ equivariant under a representation $\rho$ of $\pi_1(S)$ into $PSL(2,\C)$.

Section \ref{ssc:fuchsian} already provides us with one canonical such immersion: the one for which $\rho$ takes values in $PSL(2,\R)$. A positive answer to Question $W^{imm}_{\HH^3}$ for smooth surfaces would provide a natural parameterization of the space of equivariant isometric immersions of a closed surface $S$ equipped with a metric $h$ of curvature $K>-1$ in $\HH^3$, with the natural parameter being the pull-back on $S$ by the Gauss map of the conformal class at infinity of $\HH^3$.

\subsection{Similar statements for $K$-surfaces} \label{ssc:similar}

There are several statements concerning immersed, locally convex surfaces that can be proved, in the special case of surfaces of constant curvature, using recent results on the existence and uniqueness of an extension to the hyperbolic disk of a quasi-symmetric homeomorphism from $\RP^1$ to $\RP^1$. We collect those statements and their proofs in this section.

\subsubsection{Statements}

The first statement involves the gluing map at infinity between the first and third fundamental for on an immersed $K$-surface.

\begin{theorem} \label{tm:I-III}
  Let $K\in (-1,0)$, and let $\sigma:\RP^1\to \RP^1$ be a quasi-symmetric homeomorphism. There exists a unique immersion $u:\DD\to \HH^3$ with bounded principal curvatures such that the induced metric has constant curvature $K$ and that the gluing map at infinity between $I$ and $\III$ is equal to $\sigma$.
\end{theorem}

A pair of similar statements deal with the gluing map at infinity between the induced metric (resp. third fundamental form) and a combination of $I$ and $\III$.

\begin{theorem} \label{tm:I+III}
  Let $K\in (-1,0)$, and let $\sigma:\RP^1\to \RP^1$ be a quasi-symmetric homeomorphism. There exists a unique immersion $u:\DD\to \HH^3$ with bounded principal curvatures such that the induced metric has constant curvature $K$ and that the gluing map at infinity between $I$ and $[(1+K)I+\III]$ is equal to $\sigma$.
\end{theorem}

\begin{theorem} \label{tm:I+III*}
  Let $K\in (-1,0)$, and let $\sigma:\RP^1\to \RP^1$ be a quasi-symmetric homeomorphism. There exists a unique immersion $u:\DD\to \HH^3$ with bounded principal curvatures such that the induced metric has constant curvature $K$ and that the gluing map at infinity between $\III$ and $[(1+K)I+\III]$ is equal to $\sigma$.
\end{theorem}

Note by comparison that Question $W^{imm}_{\HH^3}$ and Question $W^{*imm}_{\HH^3}$ can be stated in the same way as Theorem \ref{tm:I+III} and Theorem \ref{tm:I+III*} but with $[(1+K)I+\III]$ replaced by $[I+2\II+\III]$, where $\II$ denotes the second fundamental form, since this conformal class is the pull-back by the Gauss map of the conformal class on $\partial_\infty\HH^3$.

\subsubsection{Parameterization by the gluing between $I$ and $\III$}

The proof of Theorem \ref{tm:I-III} relies on the fact (observed by Labourie \cite{L5}) that if $S\subset \HH^3$ is a surface of constant curvature $K>-1$, then the identity is a minimal Lagrangian map between $(S,|K|I)$ and $(S,|K^*|\III)$, where $K^*=K/(K+1)$ is the curvature of $\III$. 

Conversely, if $h,h'$ are two hyperbolic metrics on a surface $S$ such that the identity is minimal Lagrangian between $(S,h)$ and $(S,h')$, and if $K\in (-1,0)$ is a constant, then there is a unique immersion of $S$ into $\HH^3$ such that the induced metric is $(1/|K|)h$ and the third fundamental form is $(1/|K^*|)h'$. Indeed, the fact that the identity is minimal Lagrangian between $(S,h)$ and $(S,h')$ is equivalent to the existence of a bundle morphism $b:TS\to TS$ which is self-adjoint and Codazzi for $h$, of determinant $1$, and such that $h'=h(b\cdot,b\cdot)$. One can then set
$$ B=\sqrt{K+1} b~. $$
Then $B$ is self-adjoint and Codazzi for $(1/|K|)h$, and of determinant $K+1$, so that $(h,B)$ satisfy the Gauss-Codazzi equations. 

Thanks to this correspondence, immersed $K$-surfaces with bounded principal curvatures such that the gluing map at infinity between $I$ and $\III$ is $\sigma$ are in one-to-one correspondence with quasi-conformal minimal Lagrangian maps from $\HH^2$ to $\HH^2$ which extend at infinity as $\sigma$. Thanks to \cite[Theorem 1.4]{maximal}, there is a unique such quasi-conformal minimal Lagrangian diffeomorphism, and Theorem \ref{tm:I-III} follows.

\subsubsection{Parameterization by the gluing between $I$ and $\III$}

The same argument can be used to prove Theorem \ref{tm:I+III} and Theorem \ref{tm:I+III*}, based on the fact that if $h$ and $h'$ are two hyperbolic metrics on a surface $S$ such that the identity is a minimal Lagrangian map between $(S,h)$ and $(S,h')$, then the identity maps $Id:(S,h+h')\to (S,h)$ and $Id:(S,h+h')\to (S,h')$ are harmonic maps, with opposite Hopf differential, and conversely, see e.g. \cite{maximal}.

Suppose now that $u:\DD\to \HH^3$ is an immersion with induced metric $I$ of constant curvature $K\in (-1,0)$, and let $K^*=K/(K+1)$. Then the metrics $|K|I$ and $|K^*|\III$ are hyperbolic, and the identity is minimal Lagrangian between them. So the identity map from $(S,|K|I+|K^*|\III)$ to $(S,|K|I)$ is harmonic. Scaling by a factor $1/|K|$ on both side and using that $K^*=K/(K+1)$, we find that the identity between $(S, I+(1/(1+K))\III)$ and $(S,I)$ is harmonic. In other terms (since harmonicity only depends on the conformal class in the source), the identity between $(S,(1+K)I+\III)$ and $(S,I)$ is harmonic.

Conversely, given a hyperbolic metric $h$ on the disk $\DD$ such that the identity between $\DD$ (equipped with its standard conformal class) and $h$ is harmonic, there is a unique hyperbolic metric $h'$ on $\DD$ (obtained by integrating the Schwarzian equation, see \cite{wolf:teichmuller}) such that the identity map from $\DD$ to $(\DD,h')$ is harmonic, with Hopf differential opposite to that of the harmonic map from $\DD$ to $(\DD, h)$. The identity map from $(\DD,h)$ to $(\DD,h')$ is then minimal Lagrangian. Given $K\in (-1,0)$, the construction of the previous section then associates to the pair $(h,h')$ an immersion from $\DD$ to $\HH^3$ with induced metric $(1/|K|)h$ and third fundamental form $(1/|K^*|)h'$, and by construction the identity is harmonic from $(\DD, (1+K)I+\III)$.

This construction shows that there is a one-to-one correspondence between immersed $K$-surfaces with bounded principal curvature and quasi-conformal harmonic diffeomorphism between $\DD$ and $\HH^2$. Theorem \ref{tm:I+III} follows.

Theorem \ref{tm:I+III*} is proved in the same manner, using the fact that the identity from $(\DD, [h+h'])$ to $(\DD,h')$ is also harmonic, and that it can be used to reconstruct a $K$-surface as above. 


\section{The Weyl problem for unbounded surfaces in anti-de Sitter}
\label{sc:ads}

\subsection{Basic information on AdS geometry}
\label{ssc:ads-basics}

We review very briefly in this section some basic information on anti-de Sitter (AdS) geometry. We refer the reader to e.g. \cite{bonsante2020antide} for more details.

\subsubsection{The anti-de Sitter space}
Anti-de Sitter (AdS) geometry is a Lorentzian cousin of hyperbolic geometry. The space $\AdS^3$ is a Lorentzian space of constant curvature $-1$, originally introduced by physicists as a cosmological model. It can be defined as a quadric in the 4-dimensional Minkowski space, equipped with the induced metric:
$$ \AdS^3 = \{ x\in \R^{2,2}~|~ \langle x,x\rangle=-1\}~. $$
Its fundamental group is $\Z$.
$\AdS^3$ admits a projective model, similar to the Klein model for $\HH^3$. This model represents a ``hemisphere'' of $\AdS^3$ as the interior of a one-sheeted hyperboloid in $\R^3$, with geodesics of $\AdS^3$ corresponding to line segments. Space-light geodesics in $\AdS^3$ correspond to lines in $\R^3$ which intersect the boundary hyperboloid at two points, light-like geodesics corresponds to line intersecting the hyperboloid tangentially in one point, while time-like geodesics corresponds to lines in the interior of the hyperboloid but not intersecting it. A full projective model of $\AdS^3$ can be obtained by taking a double cover of $\RP^3$, and thus in the sphere $S^3$, as the interior of a quadric of signature $(1,1)$.

$\AdS^3$ is naturally equipped with a boundary, which can be seen in the projective model of $\AdS^3$ in $S^3$. This boundary is endowed with a conformal Lorentzian structure, analog to the conformal Riemannian metric on the ideal boundary of $\HH^3$.

It follows by its definition above that the isometry group of $\AdS^3$ is $O(2,2)$. Its identity component, $SO_0(2,2)$, can be identified with $PSL(2,\R)\times PSL(2,\R)$, and $\AdS^3$ can in be fact identified isometrically with $PSL(2,\R)$ equipped with its Killing form.

\subsubsection{Globally hyperbolic anti-de Sitter spacetimes}

There is a physically relevant notion of non-complete AdS spacetimes, namely, those which are {\em globally hyperbolic compact}: they contain a closed Cauchy surface, which we always assume to be of genus at least $2$. We say that a manifolds $M$ is {\em globally hyperbolic maximal compact} (GHMC) if it is globally hyperbolic compact, and maximal (in the sense of inclusion) under this condition, that is, any isometric embedding of $M$ in a globally hyperbolic compact space is in fact an isometry.

Mess \cite{mess,mess-notes} discovered striking analogies between the geometric properties of 3-dimensional Globally Hyperbolic Maximal Compact (GHMC) anti-de Sitter spacetimes, and those of quasifuchsian hyperbolic manifolds. For this reason, GHMC AdS spacetimes are now often called ``quasifuchsian AdS spacetimes'', and we will follow this convention here.

Mess \cite{mess} proved that when $M$ is a quasifuchsian AdS spacetime with Cauchy surface a closed surface $S$, its holonomy representation $\rho:\pi_1(S)\to SO_0(2,2)$ can be written, in the decomposition of $SO_0(2,2)$ as $PSL(2,\R)\times PSL(2,\R)$, as $\rho=(\rho_L, \rho_R)$, where $\rho_L,\rho_R:\pi_1(S)\to PSL(2,\R)$ have maximal Euler number, and are therefore (by a result of Goldman \cite{goldman:topological}) holonomy representations of hyperbolic metrics on $S$, which can be called the left and right hyperbolic metrics of $M$.

He also discovered an analog for quasifuchsian AdS spacetimes of the Bers Simultaneous Uniformization Theorem: given two hyperbolic metrics on a surface, there is a unique quasifuchsian AdS spacetime having them as left and right metrics.

\subsubsection{The left and right metrics on a space-like surface}

A different perspective is given in \cite{minsurf} on those left and right hyperbolic metrics. Let $S$ be a $C^2$ space-like surface in $\AdS^3$, one can define two metrics on $S$ as:
$$ I_L = I((E+ JB)\cdot, (E+ JB)\cdot)~,~~I_R = I((E-JB)\cdot, (E-JB)\cdot)~  $$
where $E:TS\to TS$ is the identity, $B$ is the shape operator, and $J$ is the complex structure of the induced metric. It can be checked by a simple computation using the Gauss formula in $\AdS^3$ (see \cite[Lemma 3.15]{minsurf}) that $I_L$ and $I_R$ are hyperbolic metrics as soon as $S$ has induced metric of negative curvature.

Those left and right metrics can be defined using a notion of ``left'' and ``right'' projection from a space-like surface $S$ on a fixed totally geodesic plane $P_0$ in $\AdS^3$. Those left and right projections can be defined using the two foliations of the boundary quadric $\partial \AdS^3$ by families of lines, which we will call the left and right foliations. Suppose first that $S$ is replaced by a totally geodesic plane $P$. Following the lines of the left (resp. right) foliation defines a homeomorphism between $\partial P$ and $\partial P_0$, which can be shown to be a projective transformation, i.e. the boundary value of an isometry from $P$ to $P_0$, and this isometry is then the left (resp. right) projection from $P$ to $P_0$, we denote it by $\pi_{L,P}$ (resp. $\pi_{R,P}$). Now coming back to the general case where $S$ is a space-like surface, for each $x\in S$, we define
$$ \pi_L(x):=\pi_{L,P(x)}(x)~,~~\pi_R(x):=\pi_{R,P(x)}(x)~, $$
where $P(x)$ is the totally geodesic plane tangent to $S$ at $x$. If the induced metric has negative curvature, then the pull-back by the left (resp. right) projections of the induced metric on $P_0$ then turns out to be $I_L$ (resp. $I_R$), see \cite{minsurf}.

\subsubsection{Earthquakes and landslides}
\label{sssc:earthquakes}

The properties of the left and right projection above take a special (and interesting) form in two special cases: for locally convex pleated surfaces, and for $K$-surfaces, for $K\in (-\infty, -1)$.

For locally convex pleated surfaces, Mess noticed that the left and right projections are {\em earthquakes} (see e.g. \cite{thurston-earthquakes}). Moreover (see \cite[Lemma 7.7]{convexhull}) if $S$ is a complete locally convex pleated surface in $\AdS^3$, then the left (resp. right) projection $\Pi_L$ (resp. $\Pi_R$) is an earthquake map, which extends the left (resp. right) projection of $\partial S$ to $\partial P_0$.

A similar situation holds for $K$-surfaces, for $K\in (-\infty, -1)$, with earthquakes replaced by {\em landslides}, as introduced and studied in \cite{cyclic,cyclic2}. Landslides can be considered as smoother versions of earthquakes, which share most of their key properties. A $\theta$-landslide between two hyperbolic $(S,h)$ and $(S', h')$ can be defined as a diffeomorphism $u:S\to S'$ such that $u^*h'=h(b\cdot, b\dot)$, where $b:TS\to TS$ is Codazzi, $\det(b)=1$, $\tr(b)=2\cos(\theta/2)$, and $\tr(Jb)<0$. Equivalently it can be defined by the condition that
$$ h' = h((\cos(\theta/2)E+\sin(\theta/2)J\bar b)\cdot, (\cos(\theta/2)E+\sin(\theta/2)J\bar b)\cdot)~, $$
where $\bar b$ is Codazzi, self-adjoint for $h$, and of determinant $1$. (The equivalence between the two definitions can be proved by setting $b=\cos(\theta/2)E+\sin(\theta/2)J\bar b$.)

Let now $S\subset \AdS^3$ be a $K$-surface, for $K\in (-\infty,-1)$. Let $\theta\in (0, \pi)$ be such that $K=-1/\cos^2(\theta/2)$. Then $\frac 1{\cos^2(\theta/2)}I$ is a hyperbolic metric on $S$. Moreover, by the Gauss formula in $\AdS^3$,
$$ K=-1-\det(B)~, $$
and therefore 
$$ \det(B)=-K-1 = \frac 1{\cos^2(\theta/2)}-1=\tan^2(\theta/2)~, $$
and $B$ can be written as $B=\tan(\theta/2)\bar b$, where $\bar b$ is Codazzi, self-adjoint for $I$, and has determinant $1$.

Therefore the left metric can be written as
\begin{eqnarray*}
  I_L & = & I((E+JB)\cdot, (E+JB)\cdot) \\
      & = & I((E+\tan(\theta/2)J\bar b)\cdot, (E+\tan(\theta/2)J\bar b)\cdot) \\
  & = & \frac 1{\cos^2(\theta/2)} I((\cos(\theta/2)E+\sin(\theta/2)J\bar b)\cdot,(\cos(\theta/2)E+\sin(\theta/2)J\bar b)\cdot)~.
\end{eqnarray*}
As a consequence, the identity between the hyperbolic metrics $(1/\cos^2(\theta))I$ and $I_L$ is a $\theta$-landslide. Similarly, the identity map between $(1/\cos^2(\theta))I$ and $I_R$ is a $(-\theta)$-landslide. If $S$ has bounded principal curvatures, then this landslide diffeomorphism is quasi-conformal. 

Conversely, a quasi-conformal $\theta$-landslide between two copies of the hyperbolic plane defines in this manner a $K$-surface in $\AdS^3$. 

The same construction works in the limit $K=-1$, that is, for locally convex pleated surfaces in $\AdS^3$. The corresponding map is then an earthquake, and earthquakes can in this manner be considered as limit cases of landslides, see \cite{cyclic,cyclic2}. An earthquake map from $\HH^2$ to $\HH^2$ is associated to any locally convex complete pleated surface in $\AdS^3$ with bounded measured pleating lamination, see \cite{mess}.

\subsubsection{Quasicircles in $\partial \AdS^3$}

It was already mentioned above that the ``ideal'' boundary of $\AdS^3$ can be identified with a quadric of signature $(1,1)$ in $S^3$, itself identified with the double cover of $\RP^3$. This quadric has a canonical decomoposition as $\RP^1\times \RP^1$, such that for $x\in \RP^1$, $\{ x\}\times \RP^1$ and $\RP^1\times \{ x\}$ are lines in $S^3$.

Moreover, $\partial \AdS^3$ can be equipped with a conformal Lorentzian metric, such that the light-like lines are also lines in $S^3$. The boundaries of complete space-like surfaces are acausal meridians, in the sense that they are limits of graphs of homeomorphisms from $\RP^1\to \RP^1$.

Among those acausal meridians, those that are graphs of quasi-symmetric homeomorphisms play a particular roles. They are often call {\em quasi-circles} in $\partial \AdS^3$, and appear to play a similar role in AdS geometry as the ``usual'' quasi-circles in $\CP^1$ play in hyperbolic geometry, with some limited differences (see \cite{width} for such a difference).

\subsubsection{Duality}
\label{sssc:duality}

The polar duality between hyperbolic and de Sitter space, already recalled in Section \ref{ssc:duality-hyperbolic}, also appears for the anti-de Sitter space. However the duality is with $\AdS^3$ itself. We recall briefly its definition and main properties here.

Let $x\in \AdS^3$. It can be considered as point in $\R^{2,2}$, let $x^{\perp}$ be its oriented orthogonal hyperplane. Since $\langle x,x\rangle=-1$ by definition of $\AdS^3\subset \R^{2,2}$, $x^\perp$ is of signature $(2,1)$, so its intersection with $\AdS^3$ is a totally geodesic oriented space-like plane, which we denote by $x^*$.

Given an oriented space-like surface $\Sigma\subset \AdS^3$, we can define its dual as the set $\Sigma^*$ of points dual to the tangent planes of $\Sigma$. If $\Sigma$ is oriented and strictly convex, then $\Sigma^*$ is also space-like, smooth and strictly convex, and $(\Sigma^*)^*=\Sigma$.

As in $\HH^3$, the polar duality exchanges the induced metric and third fundamental form: the induced metric on $\Sigma$ corresponds under the duality to the third fundamental form on $\Sigma^*$, and conversely. 

One point that can be noted is that the left and right metrics on a space-like, strictly convex surface are exchanged by duality. Indeed, under the polar duality, $I$ is replaced by $\III$, the shape operator $B$ is replace by $B^{-1}$, while the complex structure of the induced metric $J$ is replaced by the complex structure $\bar J$ of $\III$, which is equal to $\bar J = B^{-1}JB$. The left metric of the dual surface is thus equal to:
\begin{eqnarray*}
  \III((E+\bar J B^{-1})\cdot, ((E+\bar J B^{-1})\cdot) & = & \III((E+B^{-1}J)\cdot, ((E+B^{-1}J)\cdot) \\
                                                        & = & I((B+J)\cdot, (B+J)\cdot) \\
                                                        & = & I(J(B+J)\cdot, J(B+J)\cdot) \\
  & = & I((E-JB)\cdot, (E-JB)\cdot)~, 
\end{eqnarray*}
which is the right metric on the primary surface, and conversely.

\subsection{Convex domains with a quasicircle at infinity}

This section is focused on Question $W^{\Omega}_{\AdS^3}$. This question takes a simpler form in $\AdS^3$ than in $\HH^3$ since we only consider space-like surfaces in $\AdS^3$. As a consequence, if a convex domain $\Omega\subset \AdS^3$ has boundary at infinity a disjoint union of quasicircles, this family can only contain one quasicircle, and the boundary $\partial \Omega$ in $\AdS^3$ is composed of two disks sharing the same ``ideal'' boundary.

We first recall a number of known results and some open questions concerning quasifuchsian AdS spacetimes, corresponding to the special case of Question $W^{\Omega}_{\AdS^3}$ where the data on the boundary surfaces are invariant under a pair of cocompact actions of a surface group, and the gluing map at infinity is also equivariant under this pair of actions. The next section presents some recent results as well as open questions without group actions, in particular concerning metrics of constant curvature. In the last part we restate some of the same questions for surfaces of constant curvature in terms of fixed points of earthquakes and of landslides.

\subsubsection{Quasifuchsian AdS spacetimes}

Let $\Omega\subset \AdS^3$ be a convex domain invariant under a surface group action $\rho:\pi_1S\to SO_0(2,2)$ which acts cocompactly on a Cauchy surface, and such that $\rho$ acts properly discontinuously on $\Omega$. The quotient $\Omega/\rho(\pi_1S)$ is then isometric to a geodesically convex subset $\bar \Omega$ in a quasifuchsian AdS spacetime $M$.

Such a quasifuchsian spacetime $M$ contains a smallest non-empty geodesically convex subset, its {\em convex core} $C(M)$. If $M$ is not Fuchsian, $C(M)$ has non-emtpy interior, and its boundary is the disjoint union of two pleated space-like surfaces $\partial_\pm C(M)$, both with hyperbolic induced metrics $m_\pm$ pleated along a measured lamination $l_\pm$. Mess extended to this AdS setting Conjecture \ref{cj:thurston-m} as well as ``dual'' question of Thurston asking for a parameterization of the space of quasifuchsian manifolds by the measured bending lamination, solved in \cite{bonahon-otal,uniqueness}.

\begin{conjecture}[Mess] \label{cj:mess-m}
    Let $m_-, m_+$ be two hyperbolic metrics on a closed surface $S$ of genus at least $2$. Is there a unique quasifuchsian AdS structure on $S\times \R$ such that the induced metrics on the boundary components of the convex core are isotopic to $m_-$ and $m_+$? 
\end{conjecture}

\begin{conjecture}[Mess] \label{cj:mess-l}
    Let $l_-, l_+$ be two measured laminations that fill on a closed surface $S$ of genus at least $2$. Is there a unique quasifuchsian AdS structure on $S\times \R$ such that the measured bending laminations on the boundary of the convex core are $l_-$ and $l_+$?
\end{conjecture}

The first of those conjecture corresponds to a special case of Question $W^{thin}_{\AdS^3}$, while the second corresponds to a special case of Question $W^{*thin}_{\AdS^3}$.

The existence part of Conjecture \ref{cj:mess-m} was proved by Diallo, see \cite{diallo2013} or \cite[Appendix]{convexhull} but the uniqueness remains open. Similarly, the existence part of Conjecture \ref{cj:mess-l} was proved in \cite{earthquakes}, but the uniqueness is still unknown.

If $\Omega$ is a geodesically convex domain with smooth, strictly convex boundary, its boundary is the disjoint union of two space-like surfaces, and each is equipped with a Riemannian metric of curvature $K<-1$ (see \cite[Question 3.5]{adsquestions}). Tamburelli \cite{tamburelli2016} proved that any pair of metrics of curvature $K<-1$ on a closed surface can be realized in this manner. However uniqueness remains elusive here, too. Thanks to the duality recalled in Section \ref{sssc:duality}, it follows that any two metrics of curvature $K<-1$ on a closed surface can be realized as the third fundamental form on the boundary of a geodesically convex subset of a quasifuchsian AdS spacetime, therefore providing a partial answer to Question $W^{*\Omega}_{\AdS^3}$.

\subsubsection{Existence results for domains with a quasicircle at infinity}
\label{sssc:existence}

Without the asumption of a surface group acting, there are some recent existence results but without uniqueness so far.

The main result of \cite{convexhull} is that the existence part of Question $W^{\Omega}_{\AdS^3}$ has a positive answer for metrics of constant curvature $K\leq -1$: given $K\leq -1$ and a quasi-symmetric homeomorphism $\sigma:\RP^1\to \RP^1$, there exists a convex domain $\Omega\subset \AdS^3$ with boundary at infinity a quasi-circle, such that the induced metrics on the two boundary components of $\Omega$ have constant curvature $K$, and the gluing map at infinity between them is equal to $\sigma$. Recently, this result has been extended by A. Mesbah \cite{mesbah2024} to metrics of variable curvature, under the hypothesis of bounded derivatives (see Definition \ref{df:bounded-der}).

\begin{theorem}[A. Mesbah]
  Let $h_-, h_+$ be two complete, conformal metrics on $\DD$, with curvature $K\in [-1/\epsilon, -1-\epsilon]$ (for some $\epsilon>0$) and bounded derivatives. Let $\sigma:\partial\DD\to \partial \DD$ be a quasi-symmetric homeomorphism. There exists a convex domain $\Omega\subset \AdS^3$, with boundary at infinity a quasicircle, such that the induced metrics on the two boundary components of $\partial\Omega$ are $h_-, h_+$, with gluing map at infinity $\sigma$.
\end{theorem}

Thanks to the duality in Section \ref{sssc:duality}, this also implies a similar existence result concerning the third fundamental form on $\partial \Omega$, for metrics of constant curvature $K<-1$. This does not however cover the universal version of Conjecture \ref{cj:mess-l}, for which a partial result is provided in \cite{laminations}. To state it, we need two definitions.

\begin{definition} \label{df:fill}
  Let $\lambda$ and  $\mu$ be two mesured laminations on  $\HH^2$. We say that $\lambda$ and $\mu$ \textit{strongly fill} if, for any $\varepsilon>0$, there exists $c>0$ such that, if $\gamma$ is a geodesic segment in $\mathbb{H}^2$ of length at least $c$, 
\[i(\gamma,\lambda)+i(\gamma,\mu)\geqslant \varepsilon~.\]
\end{definition}

\begin{definition}
  A {\em parameterized quasicircle} in $\partial \AdS^3$ is a map $u:\RP^1\to \partial \AdS^3$ such that, under the identification of $\partial \AdS^3$, the composition on the left of $u$ with either the left or the right projection is quasi-symmetric.
\end{definition}

The following statement is the main result of \cite{laminations}.

\begin{theorem} \label{tm:bending}
Let $\lambda_-, \lambda_+\in \cML$ two bounded measured laminations that strongly fill. There exists a parameterized quasicircle $u:\mathbb{RP}^1\to \partial \AdS^3$ such that the measured bending laminations on the upper and lower boundary components of $CH(u(\RP^1))$ are $u_*(\lambda_+)$ and $u_*(\lambda_-)$, respectively.
\end{theorem}

\begin{question} \label{q:unique}
  In this setting, is $u$ unique? 
\end{question}


\subsubsection{Relations to fixed points of landslides}


We first recall the well-known relations between pleated surfaces in $\AdS^3$ and earthquakes.

Pleated surfaces in $\AdS^3$ are related to earthquakes, and the results and questions on prescribing the bending laminations on the boundary on convex hulls of quasi-circles can be stated in terms of fixed points of earthquakes. This equivalence was already noted (and used) in \cite{earthquakes} for the bending lamination on the boundary of the convex core of quasifuchsian AdS spacetimes. A similar equivalence holds without group actions, for convex hulls of quasi-circles in $\AdS^3$, and is stated in \cite{laminations}. The statement relies on a definition. 

\begin{definition} 
  We denote by $\cE^l:\cML\times \cQS\to \cQS$ the map defined as
  $$ \cE^l(\lambda)(u)=E^l(u_*\lambda)\circ u~, $$
  and similarly for $\cE^r$.
\end{definition}

Defined in this way, the map $\cE^l$ satisfies a ``flow'' property, which follows more or less directly from the definition:
$$ \cE^l(u_*(t\lambda))(\cE^l(s\lambda)(u))=\cE^l((s+t)\lambda)(u)~. $$
There is of course a similar definition for right earthquake map $\cE^r$, based on the right earthquakes $E^r$. It then follows again more or less directly from the definitions that $\cE^r$ is the inverse of $\cE^l$, in the sense that
$$ \cE^r(u_*\lambda)(\cE^l(\lambda)(u))=u~. $$

Theorem \ref{tm:bending} can then be stated equivalently as follows.

\begin{theorem}\label{tm:earthquakes}
  Let $\lambda_-, \lambda_+\in \cML$ be two bounded laminations that strongly fill. There exists a quasi-symmetric homeomorphism $u:\RP^1\to \RP^1$ such that $\cE^l(\lambda_l)(u)=\cE^r(\lambda_r)(u)$.
\end{theorem}

Question \ref{q:unique} can then also be formulated equivalently as the uniqueness of this quasi-symmetric homeomorphism $u$.

Thanks to the ``cycle'' properties above, this theorem can be stated as the existence of fixed point of the map $u\mapsto \cE^l(u_*\lambda_r)\circ\cE^l(\lambda_l)(u)$, as a map from $\cT$ to $\cT$. The uniqueness in Question \ref{q:unique} is equivalent to the uniqueness of this fixed point.


\medskip

We now consider the similar relation between $K$-surfaces and landslides.
For $K$-surfaces, a similar statement can be made, with earthquakes replaced by landslides. However measured laminations have to be replaced by quasi-symmetric maps. The construction requires some care, and begins with the definition of the landslides on $\cQS\times \cQS$, generalizing (in a not completely direct way) the definition in \cite{cyclic} for closed surfaces. We recall the definition here, basically repeating material from \cite[\S 8]{cyclic}, with notations a bit better adapted to our purpose.

Let $u,u^*\in \cQS$. Then $u^*\circ u^{-1}$ is quasi-symmetric. It follows (using the main result in \cite{maximal}) that there exists a unique quasi-conformal minimal Lagrangian diffeomorphism $w:\HH^2\to \HH^2$ such that $\partial w=u^*\circ u^{-1}$.

There is then a unique bundle morphism $b:T\HH^2\to T\HH^2$ which is self-adjoint for $h_0$, Codazzi, of determinant $1$, and such that $w^*h_0=h_0(b\cdot, b\cdot)$.

Let $\theta \in (0,2\pi)$. We set
$$ \beta_\theta = \cos(\theta/2) E + \sin(\theta/2) Jb~. $$
A direct computation (see \cite{cyclic}) shows that $\beta_\theta$ is Codazzi and of determinant $1$, and therefore $h_\theta=h_0(\beta_\theta\cdot, \beta_\theta\cdot)$ is hyperbolic. As a consequence, the identity map between $\HH^2$ and $(\HH^2,h_\theta)$ determines a quasiconformal diffeomorphism from $\HH^2$ to $\HH^2$, well-defined up to post-composition by a hyperbolic isometry. Taking its boundary value determines a quasi-symmetric map $v_\theta$. 

Note that composing $u$ or $u^*$ on the left by a M\"obius transformation does not change $b$, and therefore does not change $v_\theta$, which is defined up to post-composition by a hyperbolic isometry anyway. So $v_\theta$ is well-defined as an element of the universal Teichm\"uller space $\cT$, and depends on $u$ and $u^*$ considered as elements of $\cT$.

\begin{definition}
  Let $u,u^*\in \cT$, we set $L_{e^{i\theta}}(u,u^*)=(v_\theta u,v_{\theta+\pi}u)\in \cT\times \cT$, and $L^1_{e^{i\theta}}(u,u^*)=v_\theta u\in \cT$.
\end{definition}

As an example, for $\theta=0$, $\beta_0=E$, and as a consequence $w_0=Id$ and $L^1_1(u,^*)=u$. On the other hand, for $\theta=\pi$, $\beta_\pi=Jb$, so that $h_\pi=h_0(Jb\cdot, Jb\cdot)=h_0(b\cdot,b\cdot)=w^*h_0$, $v_\pi=w$, and $L^1_{-1}(u,u^*)=u^*$.




This definition is strongly related to the definition given in \cite{cyclic} for hyperbolic metrics on closed surfaces. Specifically, consider a closed surface $S$ of genus at least $2$, equipped with a fixed hyperbolic metric $h_0\in \cT_S$, see \cite[Prop. 8.4]{cyclic}.



  

The arguments and computations in \cite[\S 1.6]{cyclic} can be used to prove that the landslide map $L_{e^{i\theta}}$ on $\cT\times \cT$  satisfies the flow property, that is, $L_{e^{i\theta}}\circ L_{e^{i\theta'}}=L_{e^{i(\theta+\theta')}}$ for all $\theta,\theta'\in \R$.


We can now formalize the relation between landslides and $K$-surfaces in $\AdS^3$.  Let $u_L, u_R:\RP^1\to \RP^1$ be quasi-symmetric homeomorphism, and let $C$ be the graph of $u_R\circ u_L^{-1}$, seen as a quasicircle in $\partial \AdS^3$. Note that $C$ can be identified with $\RP^1$ by choosing any parameterization $\rho:\RP^1\to C$ such that $\rho$ composed with the left projection is a quasi-symmetric homeomorphism. We use such an identification below implicitly, so as to simplify notations.

Let $K\in (-\infty, -1)$, and let $S_+$ (resp. $S_-$) be the unique past-convex (resp. future-convex) $K$-surface with boundary $C$. Let $u_L:C\to \RP^1$ and $u_R:C\to \RP^1$ be the left and right projections, and let $u_+:C\to \RP^1$ (resp. $u^*_+:C\to \RP^1$) be the identification of $C$ with $\partial_\infty \HH^2$ obtained by considering the developing map of the hyperbolic metric $|K| I$ (resp. $|K^*|\III$) on $S_+$. Finally, let $u_-, u^*_-$ be the corresponding maps for $S_-$. Finally, let $\theta \in (0,\pi)$ be such that $K=-1/\cos^2(\theta/2)$. The following statement then collects the relations recalled in Section \ref{sssc:earthquakes}.


\begin{theorem}
  Under those conditions,
  $$ u_L=L^1_\theta(u_+, u_+^*)=L^1_{-\theta}(u_-, u_-^*)~, $$
  $$ u_R=L^1_{-\theta}(u_+, u_+^*)=L^1_{\theta}(u_-, u_-^*)~. $$
Conversely, if $u_L, u_R, u_+, u_+^*, u_-, u_-^*$ are elements of $\cT$ satisfying those two equations, then they arise from the construction above.
\end{theorem}










\subsection{Spacelike disks in $\AdS^3$}

We now turn to Question $W^{imm}_{\AdS^3}$ and its dual. Following the organisation of the previous section, we first focus on quasifuchsian AdS spacetimes, corresponding to immersions of the universal cover of a surface which are equivariant under a representation of the fundamental group of the surface in $SO_0(2,2)$. In this case, Question $W^{imm}_{\AdS^3}$ reduces to the following.

\begin{question}
  Let $h$ be a Riemannian metric of curvature $K<-1$ on a closed surface $S$ of genus at least $2$, and let $h_0\in \cT_S$ be a hyperbolic metric on $S$. Is there a unique quasifuchsian AdS spacetime $M$ containing a past-convex Cauchy surface with induced metric isotopic to $h$ and left metric isotopic to $h_0$? 
\end{question}

Tamburelli \cite[Prop. 7.1]{tamburelli2016} shows that existence holds in this question. It then follows from the duality described in Section \ref{sssc:duality} that the same result holds with the third fundamental form prescribed instead of the metric. 

Back to the universal case, a full answer to Question $W^{imm}_{\AdS^3}$ can be given for $K$-surfaces for $K\leq -1$. For $K=-1$, that is, for locally convex pleated surfaces, this answer is a direct consequence of Thurston's Earthquake Theorem.

\begin{proposition}
  Let $\sigma:\RP^1\to \RP^1$ be a quasi-symmetric homeomorphism. There is a unique past-convex pleated surface in $\AdS^3$ such that the gluing map at infinity between the induced metric and the left metric is $\sigma$. 
\end{proposition}

The proof follows from the fact (see \cite[Prop. 7.7]{maximal}) that given a past-convex pleated surface $S\subset \AdS^3$, the left projection from the boundary curve to a circle in $\partial \AdS^3$ (the boundary of a totally geodesic space-like plane $P$) is the continous extension to the boundary of the left projection from $S$ to $P$, which itself is the earthquake along the measured bending lamination on $S$. But any quasi-symmetric homeomorphism from $\RP^1$ to $\RP^1$ is the continuous extension to the boundary of the earthquake along a unique bounded measured lamination on the hyperbolic disk, and the result follows.

Similarly, for $K$-surfaces for $K<-1$, a positive answer to Question $Q^{imm}_{\AdS^3}$ follows from a recent result of Bonsante and Seppi \cite{bonsante-seppi:area}, who proved that any quasi-symmetric homeomorphism from $\RP^1$ to $\RP^1$ is the boundary of a unique quasi-conformal $\theta$-landslide.

\begin{proposition}
    Let $\sigma:\RP^1\to \RP^1$ be a quasi-symmetric homeomorphism, and let $K<-1$. There is a unique past-convex $K$-surface in $\AdS^3$ such that the gluing map at infinity between the induced metric and the left metric is $\sigma$. 
\end{proposition}

Again the result follows from the fact following facts.
\begin{enumerate}
\item For such a $K$-surface, the gluing map between the induced metric and the left metric is given by the left projection from the boundary of $S$ to the boundary of a totally geodesic space-like plane $P\subset \AdS^3$ (see \cite[Lemma 3.18]{maximal}).
\item For a $K$-surface, the left projection is a $\theta$-landslide, for $K=-1/\cos^2(\theta)$, and conversely if the left projection is a $\theta$-landslide, then $S$ is a $K$-surface, see above.
  \item It is proved in \cite{bonsante-seppi:area} that any quasi-symmetric homeomorphism from $\RP^1$ to $\RP^1$ extends uniquely as a quasiconformal $\theta$-landslide from $\HH^2$ to $\HH^2$. 
\end{enumerate}

\medskip

For surfaces invariant under a surface group -- or, equivalently, past-convex surfaces in a quasifuchsian AdS spacetime -- an existence and uniqueness result is obtained in \cite{immersed_weyl_ads}.

\begin{theorem} \label{tm:imm-weyl-ads}
  Let $S$ be a closed surface of genus at least $2$, let $h$ be smooth metric on $S$ of curvature $K\in [-1/\epsilon,-\epsilon]$, and let $h_0$ be a hyperbolic metric on $S$. There exists a unique quasifuchsian AdS spacetime containing a past-convex surface homeomorphic to $S$, with induced metric $h_+$ and left metric isotopic to $h_0$.  
\end{theorem}

\subsection{Ideal and hyperideal polyhedra}

There are other recent results concerning versions of the Weyl problem in $\AdS^3$, but for polyhedra rather than smooth surfaces. We would like to briefly mention the recent results in \cite{idealpolyhedra} showing that ideal polyhedra in $\AdS^3$ are uniquely determined by their induced metrics, or by their dihedral angles (which play the role of third fundamental form in this setting). Similar results hold for hyperideal polyhedra in $\AdS^3$ \cite{hyperideal}.

\section{Minkowski domains of dependence and half-pipe geometry}
\label{sc:minkowski}

We consider in this section analogs of the questions considered above in hyperbolic and anti-de Sitter space, but in Minkowski space instead. Minkowski space is naturally dual to a 3-dimensional space introduced by Jeff Danciger, see \cite{danciger:transition,danciger:ideal}, and called half-pipe space. Half-pipe geometry can be naturally considered as a ``transitional'' geometry between hyperbolic and anti-de Sitter geometry, and most of the questions considered above in the hyperbolic and anti-de Sitter settings also make sense in half-pipe, and by duality, in Minkowski space. 

Space-like surfaces in half-pipe space can be considered as first-order deformations of a totally geodesic space-like plane in $AdS^3$, or of a totally geodesic plane in $\HH^3$, see Section \ref{ssc:deformation}.

\subsection{Half-pipe space as a dual of Minkowski space}
\label{ssc:dual-HP}

We provide here some very basic definitions, see e.g. \cite{danciger:transition,bonsante-seppi_spacelike} for more details.

Half-pipe space $\HP^3$ can be defined simply as $\HH^2\times \R$, equipped with the degenerate metric $h_0+0dt^2$, where $h_0$ is the standard metric on $\HH^2$. It is naturally the geometric structure occuring in the convex subset bounded by a degenerate quadric in $\RP^3$ (or equivalently a cylinder in $\R^3$) through the Hilbert metric. In this way, $\HP^3$ is equipped with a natural notion of lines and planes, which are just the intersection with the cylinder of lines and planes in $\R^3$. A line or planed is called space-like if it is not parallel to the generatrix of the cylinder, that is, if the metric induced from that of $\HP^3$ is non-degenerate.

There is a natural duality between $\HP^3$ and Minkowski space $\R^{2,1}$. To any space-like plane $p$ in $\R^{2,1}$ one can associate its unit future-oriented normal vector $n$, and the oriented  distance $d$ between $0$ and $p$ along the line directed by $n$ through $0$. The pair $(n,d)$ defines a point in $\HP^3$, which we call the dual of $p$ and denote by $p^*$. Conversely, given a point $q\in \R^{2,1}$, the set of points $p^*\in \HP^3$ dual to the space-like planes $p\ni q$ in $\R^{2,1}$ form a space-like plane in $\HP^3$, which we call the plane dual to $q$ and denote by $q^*$. 

The isometry group of $\HP^3$ is infinite-dimensional (a fact that follows from the degeneraty of the metric). However $\HP^3$ has a restricted group of isometries, those that preserve the totally geodesic planes. This restricted isometry group is finite-dimensional. 

The identity component of the (restricted) isometry group of both $\HP^3$ and $\R^{2,1}$ can be identified to $PSL(2,\R)\ltimes \R^{2,1}$, and the action of the (restricted) isometry group commutes with the duality. 

A surface $S\subset \HP^3$ is called space-like if the induced metric on $S$ is everywhere non-degenerate. If $S$ is smooth, one can then define its second fundamental form, by considering it locally as the graph of a function over its tangent plane $p$ (each point of $S$ being associated to the point of $p$ on the same degenerate line of $\HP^3$), and taking the Hessian of that function. One can then define its Weingarten operator and third fundamental form $\III$ in the usual manner.

Let now $S\subset \R^{2,1}$ be a future-convex space-like surface with positive definite second fundamental form. One can consider the set of points in $\HP^3$ dual to the planes tangent to $S$. This turns out to be a locally strictly convex surface $S^*$. Moreover this duality exchanges the induced metric and third fundamental form, as seen for hyperbolic space in Section \ref{ssc:duality-hyperbolic} and for the anti-de Sitter space in Section \ref{sssc:duality}. (This duality can of course be seen as a limit case of the duality seen in Section \ref{sssc:duality}, see below in Section \ref{ssc:deformation}.)

\subsection{Half-pipe surfaces as first-order deformations of hyperbolic planes}
\label{ssc:deformation}

We have already seen in the introduction how a surface in $\R^3$ can be seen as a limit of surfaces in $\HH^3$ of diameter converging to $0$. In a similar manner, a surface in $\HH^3$ can be seen as a first-order deformation of a totally geodesic plane in $\HH^3$, or of a totally geodesic space-like plane in $AdS^3$. We briefly describe here (following \cite{danciger:ideal,danciger:transition}) the construction for a totally geodesic plane in $\AdS^3$, it can be adapted almost verbatim, with some small simplifications, to a totally geodesic space-like plane in $\HH^3$. 

Let $P_0$ be a totally geodesic plane in $\AdS^3$, and let $v=\nu n$ be a normal vector field defined along $P_0$, corresponding to a first-order deformation of $P$. For each $x\in P_0$ and $t\in [0,1]$, we define $\phi_t(x)=\exp_x(tv)$, where $\exp_x$ is the exponential map at $x$. This defines a one-parameter family $(\phi_t)_{t\in [0,1]}$ of embeddings of $P_0$ in $\AdS^3$.

We then call $\Omega\subset\AdS^3$ the union of geodesics orthogonal to $P_0$ --- this subset is the future cone of the point dual to $P_0$ in $\AdS^3$. For each $x\in \Omega$, we denote by $\pi(x)\in P_0$ its orthogonal projection on $P_0$, and by $t(x)$ the time-oriented distance from $\pi(x)$ to $x$ along the geodesic joining them, which is therefore orthogonal to $P_0$ at $\pi(x)$. A direct computation then shows that the AdS metric on $\Omega$ can be written as
$$ g_1(\dot x,\dot x) = \cos(t(x))^2 h_0(d\pi(\dot x),d\pi(\dot x)) - dt(\dot x)^2~, $$
where $h_0$ is the induced metric on $P_0$.

We then for $s\in (0,1)$ perform the change of variable $\tau(x)=t(x)/s$, which leads to a one-parameter family of rescaled metrics $(g_s)_{s\in (0,1]}$  defined as
$$ g_s(\dot x,\dot x) = \cos(s \tau(x))^2 h_0(d\pi(\dot x),d\pi(\dot x)) - s^2 d\tau(\dot x)^2~. $$
Clearly those metrics converge, as $s\to 0$, to the metric on $\HP^3$.

For each $s\in (0,1]$, we can now associate to the normal deformation field $v$ the surface $S_s=\phi_s(P_0)$. In the coordinates $(\pi(x),\tau(x))$ defined above, $S_s$ is the graph of the function $v:P_0\to \R$. So as $s\to 0$, $S_s$ converges to the graph of this function $v$ over $P_0$, considered as a surface $S_{HP}$ in $HP^3$.

Note that a similar description can be used to describe the dual surface $S^*$ of $S$ in $\R^{2,1}$. Indeed, the dual of $P_0$ is a point $P_0^*$, and as $t\to 0$, $S_t^*$ converges to the past light cone of  $P_0^*$. To obtain the surface dual to $S_{HP}$, one needs only to apply a one-parameter family of homotheties (or scaling) centered at $P_0^*$. 

\subsection{Quasifuchsian HP manifolds}
\label{ssc:quasifuchsian-HP}

Once half-pipe surfaces are considered as a first-order deformations of totally geodesic planes in $\AdS^3$ or $\HH^3$, one can consider equivariant first-order deformations of a totally geodesic plane as ``Cauchy surfaces'' in a half-pipe manifold. We do not elaborate on this point here, and refer to Danciger's work \cite{danciger:transition} for a detailed construction. We will call those half-pipe manifolds ``quasifuchsian'', since a number of their properties are similar or analoguous to those of quasifuchsian hyperbolic manifolds or AdS spacetimes.
\begin{itemize}
\item Their holonomy representation takes value in $PSL(2,\R)\ltimes \R^{2,1}$.
\item They contain a smallest non-empty geodesically convex subset, their convex core. Except in the ``Fuchsian'' case -- for manifolds containing a totally geodesic closed surface -- the convex core has non-empty interior, and its boundary is the disjoint union of two hyperbolic surfaces pleated along a measured bending lamination.
\item They contain a unique minimal Cauchy surface (i.e. a Cauchy surface for which the trace of $\II$ with respect to $I$ vanishes). 
\end{itemize}

There is also a notion of Minkowski dual of a quasifuchsian half-pipe manifold. However this dual is not a single Minkowski spacetime, but rather a pair of globally hyperbolic maximal compact Minkowski spacetimes, one future-complete and one past-complete, which share the same holonomy (which is of course also the holonomy of their dual half-pipe manifold). 

Note that the questions considered above concerning the induced metrics on surfaces do not work well in half-pipe geometry. This is because the induced metric on any complete space-like surface in $\HP^3$ is isometric to the hyperbolic plane, while all Cauchy surfaces in a quasifuchsian half-pipe manifold have the same induced metric. However the dual questions concerning the third fundamental forms on those surfaces do make sense. Equivalently, one can consider questions concerning the induced metrics on locally convex surfaces, but not really those concerning their third fundamental forms.

The natural analog of Question $W^{*\Omega}_{\HH^3}$ in this half-pipe setting, for domains with smooth boundary and a quasicircle at infinity, is the following.

\begin{namedquestion}[$W^{*\Omega}_{\HP^3}$]
  Let $h_-, h_+$ be two complete, conformal metrics of curvature $K<0$ on the disk $\DD^2$, and let $\sigma:\partial\DD^2\to \partial\DD^2$ be a quasi-symmetric homeomorphism. Is there a unique geodesically convex domain in $\HP^3$ with boundary at infinity a quasicircle, such that the third fundamental forms on the past and future boundary components are isometric respectively to $h_-$ and $h_+$, with gluing at infinity given by $\sigma$? 
\end{namedquestion}

This question can also be stated dually in terms of the induced metrics on space-like surfaces in pairs of Minkowski domains of dependence. The statement is somewhat less visual than that of Question $W^{*\Omega}_{\HP^3}$. Given a pair of corresponding domains of dependence $D_-, D_+$ in $\R^{2,1}$, one past-complete and one future-complete, and given two complete space-like surfaces $S_-\subset D_-$ and $S_+\subset D_+$, there is a natural identication between asymptotic directions on $S_-$ and on $S_+$, obtained by identifying points at infinity corresponding to parallel half-lines in $\partial D_-$ and $\partial D_+$. Question $W^{*\Omega}_{\HP^3}$ asks whether given any two complete conformal metrics on $\DD^2$ of negative curvature, and a quasi-symmetric homeomorphism $\sigma:\partial\DD^2\to \partial\DD^2$, there is a unique pair of corresponding domains of dependence $D_-, D_+$ containing complete space-like surfaces $S_-, S_+$ with induced metrics isometric to $g_-$ and $g_+$,respectively, with identification at infinity given by $\sigma$. 

\subsection{Equivariant surfaces}

The special case of Question $W^{*\Omega}_{\HP}$ was recently proved by Graham Smith \cite{smith2020weyl}. His result can be stated as follows.

\begin{theorem}[G. Smith] \label{tm:smith}
  Let $S$ be a closed surface, and let $g_-, g_+$ be two metrics of negative curvature on $S$. There is then a unique quasifuchsian half-pipe manifold containing a geodesically convex subset such that the third fundamental forms on the past and future boundary components are isotopic to $g_-$ and $g_+$, respectively. 
\end{theorem}

Smith in fact states and proves the dual statement: given a closed surface $S$ and a pair $(g_-, g_+)$ of negatively curved metrics on $S$, there exists a unique GHMC Minkowski spacetime into which $(S,g_-)$ and $(S,g_+)$ isometrically embed as Cauchy surfaces in the past and future components respectively. 

An analogous result for polyhedral surfaces was the obtained by Fillastre and Prosanov \cite{fillastre-prosanov}. The metrics which are obtained are then Euclidean metrics with cone singularities of angle larger than $2\pi$, and the authors prove that each pair of such metrics can be uniquely obtained as third fundamental forms (or dual metrics, as they are called in this polyhedral setting). Again the result can be stated in terms of pairs of polyhedral surfaces in pairs of GHMC Minkowski spacetimes (sharing the same holonmy representation) with given induced metrics.

Note also that the corresponding statements for pleated surfaces in half-pipe manifolds, where one prescribes the measured pleating lamination, is proved in \cite[Theorem B.2]{earthquakes}, see also \cite{bonahon-almost}. 

\subsection{Pairs of unbounded surfaces and minimizing diffeomorphisms}

The proof of Theorem \ref{tm:smith} in \cite{smith2020weyl} (translated into half-pipe geometry) is based on the fact that if $S_-$ and $S_+$ are future-convex and past-convex surfaces in a quasifuchsian HP manifold $M$, with third fundamental forms $g_-$ and $g_+$, respectively, then the natural projection $\pi:S_+ \to S_-$ obtained by following the degenerate lines in $M$ can be factored as
$$ \pi = \pi_-\circ (\pi_+)^{-1}~, $$
where $\pi_\pm$ are the identity maps from $(S_\pm,I)$ to $(S_\pm, \III)$ and $(S_-,I)$ is identified isometrically to $(S_+,I)$. Moreover both $\pi_-$ and $\pi_+$ are {\em minimizing} in the sense of Trapani and Valli \cite{trapani-valli}, who proved the existence and uniqueness of such a minimizing map isotopic to the identity between two negatively curved metrics on a surface. 

Question $W^{*\Omega}_{\HP^3}$ thus leads quite naturally to the question of the existence and uniqueness of a minimizing map with given (quasi-symmetric) boundary behavior between two complete conformal metrics of negative curvature on the disk $\DD^2$ -- a result which is known \cite{maximal} only for metrics of constant curvature $-1$. 

\section{Further questions: the Weyl problem and $K$-surfaces in higher Teichm\"uller theory}

There has been in the last years considerable interest in extending elements of Teichm\"uller theory to specific representations in Lie groups more general than $SL(2)$. The first development in this area was by Hitchin \cite{hitchin}, with key developments by Labourie \cite{labourie-anosov,labourie-cross,labourie2009} as well as Fock and Goncharov \cite{fock-goncharov-1} (we can obviously not give proper references to all important contributions here). Several types of representations fit into this general setting, in particular Anosov representations and maximal representations (depending on the Lie group being considered). 

AdS quasifuchsian spacetimes can be considered as a ``baby case'' of this higher Teichm\"uller theory, with the Lie group being $SO(2,2)$ or $PSL(2,\R)\times PSL(2,\R)$.

As already mentioned above, it would be particularly interesting in this area to attach to each ``admissible'' representation of a surface group a ``canonical'' surface, satisfying a specific property, invariant under this surface group in the symmetric space of the Lie group, or in another associated space. A first candidate for the defining property of this surface is to consider minimal surfaces, but the existence and uniqueness of minimal surfaces in the symmetric space of the Lie group is only known for Hitchin representations of rank $2$ \cite{labourie:cyclic}, and it fails in higher rank \cite{sagman-smillie:unstable}.

\bibliographystyle{alpha}
\bibliography{/home/jean-marc/Dropbox/papiers/outils/biblio}
\end{document}